\documentclass[final]{siamonline1116}

\usepackage{lipsum}
\usepackage{amsfonts}
\usepackage{graphicx}
\usepackage{epstopdf}
\usepackage{algorithmic}
\ifpdf
  \DeclareGraphicsExtensions{.eps,.pdf,.png,.jpg}
\else
  \DeclareGraphicsExtensions{.eps}
\fi

\usepackage{amsmath}
\usepackage{amssymb}
\usepackage{verbatim}
\usepackage{mathrsfs}
\usepackage{graphicx}
\usepackage{mathtools}
\usepackage{caption}

\definecolor{darkblue}{rgb}{0,0,0.7}

\definecolor{darkgreen}{rgb}{0.01,0.75,0.24}
\newcommand\as[1]{\textcolor{darkgreen}{#1}}

 \renewcommand{\Delta}{\triangle}
\newcommand{\bee}{\bf {e}}
 
\def \Ee[#1]{\mathcal{E}^{\text{{#1}}}}
\def\R{\mathbb{R}}

\def\EE{\mathscr{E}} 

\def\pa[#1,#2]{\frac{\partial {#1}}{\partial {#2}} }

\def\idom[#1,#2,#3]{\int_{#1}\hspace{1pt} {#2} \hspace{1pt} \text{d}{#3}}
\def\res[#1,#2]{\left.{#1}\right|_{#2}}

\def\gt{\rightarrow}
\def\lgt{\downarrow}

\def\var[#1,#2]{\langle \delta \mathcal{E}^{\text{{#1}}}({#2}),v\rangle}
\def\vars[#1,#2,#3]{\langle \delta^2\mathcal{E}^{\text{{#1}}}({#2})v,{#3}\rangle}
\def\vard[#1,#2,#3,#4]{\langle \delta\mathcal{E}^{\text{{#1}}}({#2})-\delta\mathcal{E}^{\text{{#3}}}({#4}),v\rangle}

\def\P{\mathbb{P}}

\def\E{\mathbb{E}}

\def\N{\mathbb{N}}

\newcommand{\bSig}{\boldsymbol{\Sigma}}
\newcommand{\bLam}{\boldsymbol{\Lambda}}
\newcommand{\bbeta}{\boldsymbol{\beta}}

\newcommand{\bdeta}{\boldsymbol{\eta}}
\newcommand{\balpha}{\boldsymbol{\alpha}}
\newcommand{\bxi}{\boldsymbol{\xi}}

\newcommand{\bm}{\mathbf{m}}

\newcommand{\bA}{\mathbf{A}}

\newcommand{\A}{\mathcal{A}}

\newcommand{\M}{\mathcal{M}}

\newcommand{\eps}{\varepsilon}

\newcommand{\Tr}{\mathrm{Tr}}

\newcommand{\bI}{\mathbf{I}}
\newcommand{\bX}{\mathbf{X}}

\newcommand{\by}{\mathbf{y}}

\newcommand{\be}{\begin{equation}}
\newcommand{\en}{\end{equation}}
\newcommand{\ben}{\begin{equation*}}
\newcommand{\enn}{\end{equation*}}
\newcommand{\bea}{\begin{aligned}}
\newcommand{\ena}{\end{aligned}}

\def\ba#1\ena{\begin{align}#1\end{align}}
\def\ban#1\enan{\begin{align*}#1\end{align*}}

\newsiamremark{remark}{Remark}
\newsiamthm{assumption}{Assumption}

\numberwithin{theorem}{section}

\newcommand{\TheTitle}{Gaussian Approximations for Probability Measures on $\mathbb{R}^d$} 
\newcommand{\TheAuthors}{Yulong Lu, Andrew Stuart and Hendrik Weber}

\headers{\TheTitle}{\TheAuthors}

\title{{\TheTitle} \thanks{
\funding{YL is supported by EPSRC as part of the MASDOC DTC at the University of Warwick with grant No. EP/HO23364/1. AMS is supported by DARPA, EPSRC
and ONR. HW is supported by the Royal Society through the University Research Fellowship UF140187.}}}

\author{
  Yulong Lu\thanks{Mathematics Institute, University of Warwick, Coventry, CV4 7AL, UK
    (\email{yulong.lu@warwick.ac.uk}, \email{hendrik.weber@warwick.ac.uk}).}
  \and
  Andrew Stuart\thanks{Computing \& Mathematical Sciences, California Institute of Technology, Pasadena, CA 91125, USA (\email{astuart@caltech.edu}).}
  \and
  Hendrik Weber\footnotemark[2]
}

\usepackage{amsopn}

\ifpdf
\hypersetup{
  pdftitle={\TheTitle},
  pdfauthor={\TheAuthors}
}
\fi

\begin{document}

\maketitle

\begin{abstract}
 This paper concerns the approximation of probability measures on $\mathbb{R}^d$
with respect to the Kullback-Leibler divergence.
Given an admissible target measure, we show the existence of the best
approximation, with respect to this divergence, from certain sets of
Gaussian measures and Gaussian mixtures. The asymptotic behavior of
such best approximations is then studied in the small parameter limit where the measure concentrates; this asympotic
behaviour is characterized using $\Gamma$-convergence. The theory developed
is then applied to understand the frequentist consistency of Bayesian
inverse problems in finite dimensions.
For a fixed realization of additive observational noise, 
we show the asymptotic
normality of the posterior measure in the small noise limit. Taking
into account the randomness of the noise, we prove a Bernstein-Von Mises type result for the posterior measure.
\end{abstract}



\section{Introduction}

In this paper, we study the ``best'' approximation of a general finite 
dimensional probability measure, which could be non-Gaussian, from a set of 
simple probability measures, such as a single Gaussian measure or a Gaussian 
mixture family. We define ``best'' to mean the measure within the simple class 
which minimizes the Kullback-Leibler divergence between itself and the target 
measure. This type of approximation is central to many ideas, especially including the so-called ``variational inference'' \cite{wainwright2008graphical}, that are widely
used in machine learning \cite{B06}. Yet such approximation has not been the subject of
any substantial systematic underpinning theory. The purpose of this paper
is to develop such a theory in the concrete finite
dimensional setting in two ways: (i) by establishing the existence of
best approximations; (ii) by studying their asymptotic properties in
a measure concentration limit of interest. The abstract theory is then
applied to study frequentist consistency \cite{V00} of Bayesian inverse 
problems. 

\subsection{Background and Overview}

The idea of approximation for probability measures with respect to 
Kullback-Leibler divergence has been applied in a number of areas; 
see for example \cite{MG11,KP13,LSW16,sanz2016gaussian}. 
Despite the wide usage of Kullback-Leibler approximation, systematic
theoretical study has only been initiated recently. In \cite{PSSW15a}, the measure approximation problem is studied from 
the calculus of variations point of view, and existence of minimizers 
established therein; the companion paper \cite{PSSW15b} proposed 
numerical algorithms for implementing Kullback-Leibler minimization in practice.
In \cite{LSW16}, Gaussian approximation is used as a new approach for 
identifying the most likely path between equilibrium states in molecular 
dynamics; furthermore, the asymptotic behavior of the Gaussian approximation 
in the small temperature limit is analyzed via $\Gamma$-convergence. Here our 
interest is to develop the ideas in \cite{LSW16} in the context of a general 
class of measure approximation problems in finite dimensions.

To be concrete we consider approximation of a family of probability 
measures $\{\mu_\eps\}_{\eps > 0}$ on $\R^d$ with (Lebesgue)
density of the form
\be\label{eq:mu_eps}
\mu_\eps (dx) = \frac{1}{Z_{\mu,\eps}} \exp\left(-\frac{1}{\eps} V^\eps_1(x) - V_2(x)\right) dx;
\en
here $Z_{\mu, \eps}$ is the normalization constant. A typical example of a measure $\mu_\eps$ with this form is a posterior measure in Bayesian inverse problems. For instance, consider the inverse problem of identifying $x$ from a sequence of noisy observations $\{y_j\}_{j\in \N}$ where
$$
y_j  = G(x) + \eta_{j},
$$
and where the  $\eta_j$ denote describe the random noise terms. This may model a statistical measurement with an increasing number of observations or with vanishing noise. In the Bayesian approach to this inverse problem, if we take a prior with density proportional to $\exp(-V_2(x))$, then the posterior measure is given by \eqref{eq:mu_eps} with the function $\eps^{-1} V_1^\eps$, up to an additive constant, coinciding with the negative log-likelihood. The parameter 
$\eps$ is associated with the number of observations or the noise level of 
the  statistical experiment.

 Our study of Gaussian approximation to the measures $\mu_\eps$ in \eqref{eq:mu_eps} is partially motivated by the famous Bernstein-von Mises (BvM) theorem \cite{V00} in asymptotic statistics. Roughly speaking, the BvM theorem states that under mild conditions on the prior, the posterior distribution of a Bayesian procedure converges to a Gaussian distribution centered at any consistent estimator (for instance the maximum likelihood estimator)  in the limit of large 
data (or, relatedly, small noise \cite{BL}). The BvM theorem is of great importance in Bayesian statistics for at least two reasons. First, it gives a quantitative description of how the posterior contracts to the underlying truth. Second, it implies that the Bayesian credible sets are asymptotically equivalent to frequentist confidence intervals and hence the estimation of the latter can be realized by making use of the computational power of Markov Chain Monte Carlo algorithms. We interpret the BvM phenomenon in the abstract theoretical framework of best Gaussian approximations with respect to a Kullback-Leibler measure
of divergence.

\subsection{Main Contributions}

The main contributions of this paper are twofold:

\begin{itemize}

\item We use the calculus of variations to give a framework to the problem
of finding the best Gaussian (mixture) approximation of a given measure,
with respect to a Kullback-Liebler divergence;

\item We study the resulting calculus of variations problem in the small
noise (or large data) limits, therby making new links to, and ways
to think about, the classical Bernstein-von Mises theory of 
asymptotic normality.

\end{itemize}

We describe these contributions in more detail.
First we introduce a theoretical framework of calculus of variations to analyze the measure approximation problem. Given a measure $\mu_\eps$ defined by \cref{eq:mu_eps}, we find a 
measure $\nu_\eps$ from a set of simple measures, Gaussians or mixtures of finitely many Gaussians, which minimizes the 
Kullback-Leibler divergence $D_{\text{KL}}(\nu || \mu_\eps)$. We 
characterize the limiting behavior of the best approximation $\nu_\eps$ as well as the limiting behaviour of the Kullback-Leibler divergence as $\eps \lgt 0$ using the framework of  $\Gamma$-convergence. In particular, if $\mu_\eps$ is a multimodal distribution and $\nu_\eps$ is the best approximation from within the class of Gaussian mixtures, then  the limit of the minimized KL-divergence $D_{\text{KL}}(\nu_\eps || \mu_\eps)$ can characterized explicitly as the sum of two contributions:
a local term which consists of a weighted sum of the  KL-divergences between the Gaussian approximations, as well 
as the Gaussian measure whose covariance is determined by the Hessian of $V_2$ at its minimizers; and a global term which measures
how well the weights approximate the mass distribution between the modes; see \cref{thm:gamma-2}. 

We then adopt the abstract measure approximation theory to understanding the posterior consistency of finite dimensional Bayesian inverse problems. In particular, we give an alternative (and more analytical) proof of the Bernstein-von Mises theorem, see \cref{thm:expdkl} and \cref{cor:bvm}. We highlight the fact
that our BvM result improves classical BvM results for parametric statistical models in two aspects. Firstly, the convergence of posterior in the total variation distance is improved to convergence in the KL-divergence, under certain regularity assumptions on the forward map. Secondly, our BvM result allows the posterior distribution to be multimodal, in which case the posterior approaches a mixture of Gaussian distributions rather than a single Gaussian distribution in the limit of infinite data. These improvements come at a cost, and we need to make 
stronger assumptions than those made in classical BvM theory.

\subsection{Structure}

The rest of the paper is organized as follows. In \cref{sec:SU}
we set up various underpinning concepts which are used throughout
the paper: in \cref{subsec:dkl} and \cref{subsec:gamma}, 
we recall some basic facts on Kullback-Leibler divergence and 
$\Gamma$-convergence and in \cref{ssec:A} and \cref{ssec:N}
we spell out the assumptions made and the notation used. 
 In \cref{sec:SGaussian} and \cref{sec:gaussianmix} we \as{study}
the problem of approximation of the measure $\mu_\eps$ by, respectively, a single 
Gaussian measure and a Gaussian mixture. In particular, the small $\eps$
asymptotics of the Gaussians (or Gaussian mixtures) are captured by using 
the framework of $\Gamma$-convergence. In \cref{sec:app}, 
the theory which we have developed is applied to understand the posterior 
consistency for Bayesian inverse problems, and connections to the
BvM theory. 
Finally, we finish in \cref{sec:conclusion} with several conclusion remarks. 

\section{Set-Up}\label{sec:SU}

\subsection{Kullback-Leibler Divergence} \label{subsec:dkl}
Let $\nu$ and $\mu$ be two probability measures on $\R^d$ and assume
that $\nu$ is absolutely continuous with resepct to $\mu$. The 
Kullback-Leibler divergence, or relative entropy, of $\nu$ with respect
to $\mu$ is
$$
 D_{\text{KL}}(\nu || \mu) = \E^{\nu} \log \left(\frac{d\nu}{d \mu}\right).
$$
If $\nu$ is not absolutely continuous with respect to $\mu$, then the 
Kullback-Leibler divergence is defined as $+\infty$. By definition, 
the Kullback-Leibler divergence is non-negative but it is not a metric 
since it does not obey the triangle inequality  
and it is not symmetric in its two arguments. In this paper, we will 
consider minimizing $D_{\text{KL}}(\nu || \mu_\eps)$ with respect to $\nu$,
over a suitably chosen set of measures, 
and with $\mu_\eps$ being the target measure defined in \cref{eq:mu_eps}. 
Swapping the order of these two measures within the divergence is undesirable 
for our purposes. This is because minimizing $D_{\text{KL}}(\mu_\eps || \cdot)$
within the set of all Gaussian measures will lead to matching of moments \cite{B06}; this is inappropriate for multimodal measures where a more desirable
outcome would be the existence of multiple local minimizers at each
mode \cite{PSSW15a,PSSW15b}. 

Although the Kullback-Leibler divergence is not a metric, its information
theoretic interpretation make it natural for approximate inference.
Furthermore it is a convenient 
quantity to work with for at least two reasons. First the divergence provides 
useful upper bound for many metrics; in particular, one has the Pinsker 
inequality
\be\label{ieq:pinsker}
d_{\text{TV}} (\nu, \mu) \leq \sqrt{\frac{1}{2} D_{\text{KL}}(\nu || \mu)}
\en
where $d_{\text{TV}}$ denotes the total variation distance. Second the 
logarithmic structure of $D_{\text{KL}}(\cdot || \cdot)$ allows us to carry 
out explicit calculations, and numerical computations, which are
considerably more difficult when using the total variation distance directly. 

\subsection{$\Gamma$-convergence}\label{subsec:gamma}
We recall the definition and a basic result concerning
$\Gamma$-convergence. This is a useful tool for studying families
of minimization problems. In this paper we will use it to study
the parametric limit $\eps \to 0$ in our approximation problem.

\begin{definition}
\label{d:gcc}
Let $\mathcal{X}$ be a metric space and $E_\eps: \mathcal{X} \gt \R$ a family of functionals indexed by $\eps > 0$. Then $E_\eps$ $\Gamma$-converges to $E:\mathcal{X} \gt \R$ as $\eps \gt 0$ if the following conditions hold:

(i) (liminf inequality) for every $u \in \mathcal{X}$, and for every sequence
$u_\eps \in \mathcal{X}$ such that $u_\eps \gt u$, it holds that
$
E(u) \leq \liminf_{\eps \lgt 0} E_\eps(u_\eps)
$;

(ii) (limsup inequality) for every $u \in \mathcal{X}$ there exists a recovery sequence $\{u_\eps\}$ such that $u_\eps \gt u$ and $E(u) \geq \limsup_{\eps \lgt 0} E_\eps(u_\eps)$.

We say a sequence of functionals $\{E_\eps\}$ is compact if $\limsup_{\eps\lgt 0} E_\eps(u_\eps) < \infty$ implies that there exists a subsequence $\{u_{\eps_j}\}$ such that $u_{\eps_j} \gt u\in \mathcal{X}$. 
\end{definition}

The notion of $\Gamma$-convergence is useful because of the following fundamental theorem, which can be proved by similar methods as the proof of \cite[Theorem 1.21]{bra02a}. 

\begin{theorem}\label{thm:fgamma}
Let $u_\eps$ be a minimizer of $E_\eps$ with $\limsup_{\eps \lgt 0} E_\eps(u_\eps) < \infty$. If $E_\eps$ is compact and $\Gamma$-converges to $E$, then there exists a subsequence $u_{\eps_j}$ such that $u_{\eps_j} \gt u$ where $u$ is a minimizer of $E$.
\end{theorem}

Thus, when this theorem applies, it tells us that minimizers of $E$ characterize
the limits of convergent subsequences of minimizers of $E_\eps$. In other
words the $\Gamma-$limit captures the behavior of the minimization problem
in the small $\eps$ limit.

\subsection{Assumptions}\label{ssec:A}

Throughout the paper, we make the following assumptions on the potential functions $V^\eps_1$ and $V_2$ which define the target measure of interest.

\begin{assumption}\label{assump}
\item[ (A-1)] For any $\eps > 0$, $V_1^\eps$ and $V_2$ are non-negative functions in the space $C^4(\R^d)$ and $C^2(\R^d)$ respectively. Moreover, there exists constants $\eps_0 > 0$ and $M_V > 0$ such that when $\eps < \eps_0$,
$$
 \left|\partial^\alpha_x V^\eps_1(x)\right| \vee 
\left|\partial^\beta_x V_2(x)\right|
  \leq M_V e^{|x|^2}
$$
 $\text{ any } |\alpha| \leq 4, |\beta| \leq 2 \text{ and all } x\in \R^d.$

\item [(A-2)] There exists $n > 0$ such that when $\eps \ll 1$, the set of minimizers of $V^\eps_1$ is
$
\EE^\eps = \{x^1_\eps, x^2_\eps, \cdots, x^n_\eps\}
$
and $V^\eps_1(x^i_\eps) = 0, i=1,\cdots, n$. 

\item [(A-3)] There exists $V_1$ such that $V^\eps_1 \gt V_1$ pointwise. The limit $V_1$ 
has $n$ distinct global minimisers which are given by $\EE =  \{x^1, x^2, \cdots, x^n\}$. For each $i =1, \ldots, n$  the Hessian $D^2 V_1(x^i)$ is positive definite. 

\item[(A-4)] The convergence $x^i_\eps \gt x^i$ holds. 

\item [(A-5)] There exist constants $c_0, c_1 > 0$ and $\eps_0 > 0$ such that when $\eps < \eps_0$,
$$
V^\eps_1(x) \geq -c_0 + c_1 |x|^2, x\in \R^d.
$$
\end{assumption}

\begin{remark}
Conditions (A-2)-(A-4) mean that for sufficiently small $\eps > 0$, the function $V^\eps_1$ behaves like a quadratic function in the neighborhood of the minimizers $x_\eps^i$ and of $x^i$. Consequently, the measure $\mu_\eps$ is asymptotically normal in the local neighborhood of $x^i_\eps$. In particular, in conjunction with Condition (A-5) this implies that there exists $\delta> 0$ and $C_{\delta} > 0$ such that $\forall\, 0 \leq \eta < \delta$,
\be\label{eq:distV}
\text{dist}(x, \EE) \geq \eta \Longrightarrow \liminf_{\eps \lgt 0} V^\eps_1(x) \geq C_{\delta} |\eta|^2.
\en

\end{remark}
\begin{remark}
The local boundedness of $V_1^\eps$ in $C^4(\R^d)$  (Assumption (A-1)) together with the pointwise convergence of $V^{\eps}_1$ to $V_1$ (Assumption (A-3)) implies the much stronger locally uniform convergence of derivatives up to order $3$.
Furthermore, (A-4) then implies that $V^\eps_1(x^i_\eps) \gt V_1(x^i)$ and $D^2 V^\eps_1(x^i_\eps) \gt D^2 V_1(x^i)$.
\end{remark}

\subsection{Notation}\label{ssec:N}

Throughout the paper, $C$ and $\tilde{C}$ will be generic constants which 
are independent of the quantities of interest, and may change from 
line to line. Let $\mathcal{S}_{\geq}(\R, d)$ and $\mathcal{S}_{>}(\R, d)$ 
be the set of all $d\times d$ real matrices which are positive semi-definite or positive definite, respectively. 
Denote by $N(m, \bSig)$ a Gaussian measure with mean $m$ and covariance matrix 
$\bSig$. We use $|\bA|$ to denote the Frobenius norm of the $d\times d$ matrix 
$\bA$, namely $|\bA| = \sqrt{\Tr(\bA^T \bA)}$. We denote by $\lambda_{\min} (\bA)$
the smallest eigenvalue of $\bA$. 
We let $B(x, r)$ denote a ball in $\R^d$ with center $x$ and radius $r$. Given 
a random variable $\eta$, we use $\E^{\eta}$ and $\P^\eta$ when computing 
the expectation and the probability under the law of $\eta$ respectively.

\section{Approximation by Single Gaussian measures}\label{sec:SGaussian}

Let $\A$ be the set of Gaussian measures on $\R^d$, given by $$\A = \{ N(m,  \bSig): m\in \R^d, \bSig\in \mathcal{S}_{\geq}(\R, d) \}.$$
The set $\A$ is closed with respect to 
weak convergence of probability measures. 
 Consider the variational problem
\be\label{prob:var}
\inf_{\nu \in \A} D_{\text{KL}}(\nu || \mu_\eps).
\en
Given $\nu = N(m, \bSig) \in \A$, the Kullback-Leibler divergence $D_{\text{KL}}(\nu || \mu_\eps)$ can be calculated explicitly as
\be\label{eq:dkl}
\bea
D_{\text{KL}}(\nu || \mu_\eps) &= \E^{\nu} \log\left(\frac{d\nu}{d\mu_\eps}\right)\\
& = \frac{1}{\eps}\E^{\nu} V^\eps_1(x) + \E^{\nu} V_2(x) - \log \sqrt{(2\pi)^d\det \bSig} -\frac{d}{2} + \log Z_{\mu, \eps}.
\ena
\en
If $\bSig$ is non-invertible then $D_{\text{KL}}(\nu || \mu_\eps) = +\infty$. The term $-\frac{d}{2}$ comes from the expectation $\E^{\nu} \frac{1}{2} (x - m)^T\bSig (x - m)$ and is independent of $\bSig$. The term $-\log \sqrt{(2\pi)^d\det \bSig}$ prevents the measure $\nu$ from being too close to a Dirac measure.
The following theorem shows that the problem \cref{prob:var} has a solution.
\begin{theorem}\label{thm:exists1}
Consider the measure $\mu_\eps$ given by \cref{eq:mu_eps}. For any $\eps > 0$, there exists at least one probability measure $\overline{\nu}_\eps \in \A$ solving the problem \cref{prob:var}.
\end{theorem}
\begin{proof}
We first show that the infimum of \cref{prob:var} is finite. 
 In fact, consider $\nu^\ast = N(0,  \frac14 \bI_d)$. Under the \cref{assump} (A-1) we have that
$$
\E^{\nu^\ast} V^\eps_1(x) \vee \E^{\nu^\ast} V_2(x) \leq \frac{M_V}{\sqrt{(2\pi \times \frac{1}{4})^d }} \int_{\R^d} e^{-\frac{4}{2 } |x|^2 + |x|^2} dx < \infty.
$$
Note that the integral in the last expression is finite due to $-\frac{4}{2} +1<0$.
Hence we know from \cref{eq:dkl} that $\inf_{\nu \in \A} D_{\text{KL}}(\nu || \mu_\eps)
< \infty$.  
Then the existence of minimizers follows from the fact that the Kullback-Leibler divergence has compact sub-level sets and the closedness of $\A$ with respect to weak convergence of probability measures; see e.g.  \cite[Corollary 2.2]{PSSW15a}. 
\end{proof}
We aim  to understand the asymptotic behavior of the minimizers $\overline{\nu}_\eps$  of the problem \cref{prob:var} as
 $\eps \lgt 0$. Due to the factor $\frac{1}{\eps}$ in front of $V_1^{\eps}$ in the definition of $\mu_\eps$,   \cref{eq:mu_eps}, we expect the typical size of fluctuations around the minimizers to be of order $\sqrt\eps$
 and we reflect that in our choice of scaling. More precisely, for $m \in \R^d$, $\bSig \in \mathcal{S}_{\geq}(\R, d)$ we define $\nu_\eps = N(m, \eps \bSig)$ and set
\be\label{eq:Feps0}
F_\eps(m, \bSig) :=  D_{\text{KL}}( \nu_\eps  || \mu_\eps).
\en
Understanding the asymptotic behavior of minimizers $\overline{\nu}_\eps$ in
the small $\eps$ limit may be achieved by understanding $\Gamma$-convergence
of the functional $F_\eps$.

To that end, we define weights
$$
\beta^i = \left(\det D^2V_1(x^i)\right)^{-\frac{1}{2}}\cdot e^{-V_2(x^i)}, \qquad i = 1,\cdots, n,
$$
and the counting probability measure on $\{1, \ldots, n\}$ given by 
$$
\bbeta := \frac{1}{\sum_{j=1}^n \beta^j} (\beta^1, \cdots, \beta^n).
$$ 
Intuitively, as $\eps \lgt 0$, we expect the measure $\mu_\eps$ to concentrate 
on the set $\{x^i\}$ with weights on each $x^i$ given by
$\bbeta$; this intuition is reflected in the asymptotic behavior  of the normalization constant $Z_{\mu, \eps}$, as we now show. 
By definition,
$$
Z_{\mu, \eps} = \int_{\R^d} \exp\left(-\frac{1}{\eps} V^\eps_1(x) - V_2(x)\right) dx.
$$
 The following lemma follows from the Laplace approximation for integrals (see e.g. \cite{Jensen}) and \cref{assump} (A-4). 

\begin{lemma}\label{lem:normconst}
Let $V^\eps_1$ and $V_2$ satisfy \cref{assump}. Then as $\eps \lgt 0$,
\be
Z_{\mu, \eps} = \sqrt{(2\pi\eps)^d} \cdot \left(\sum_{i= 1}^n \beta^i \right) \cdot \left( 1+ o(1)\right).
\en
\end{lemma}

Recall from \eqref{eq:Feps0} that $F_\eps(m, \bSig)  = D_{\text{KL}}(\nu_\eps || \mu_\eps)$  with the specific scaling $\nu_\eps = N(m, \eps \bSig)$. In view of the expression \eqref{eq:dkl} for the Kullback-Leibler divergence, it follows from \cref{lem:normconst} that 
\be\label{eq:dkl_asym1}
F_\eps(m, \bSig) = \frac{1}{\eps}\E^{\nu_\eps} V^\eps_1(x) + \E^{\nu_\eps} V_2(x)  -\frac{d}{2}  - \frac{1}{2}\log\left(\det \bSig\right)+ \log \left(\sum_{i=1}^n \beta^i\right) + o(1).
\en

Armed with this analysis of the normalization constant we
may now prove the following theorem which identifies the $\Gamma$-limit of $F_\eps$.
To this end we define 
$$F_0(m,\bSig) := V_2(m) + \frac{1}{2}\Tr\left(D^2V_1(m)\cdot \bSig\right)-\frac{d}{2}- \frac{1}{2}\log\det\bSig  + \log \left(\sum_{i=1}^n \beta^i\right).$$

\begin{theorem}\label{thm:gamma}
The $\Gamma$-limit of $F_\eps$ is 
\be\bea\label{eq:F}
F(m, \bSig) := \begin{cases}
F_0(m,\bSig)
& \text{ if } m \in \EE \text{ and } \bSig \in \mathcal{S}_{>}(\R, d) ,\\
\infty & \text{ otherwise}.
\end{cases}
\ena
\en
\end{theorem}
The following corollary follows directly from the $\Gamma$-convergence of $F_\eps$. 
\begin{corollary}\label{cor:covmin}
Let $\{(m_\eps, \bSig_\eps)\}$ be a family of minimizers of $\{F_\eps\}$. 
Then there exists a subsequence $\{\eps_k\}$ such that $(m_{\eps_k}, \bSig_{\eps_k}) \gt (m, \bSig)$ and $F_{\eps_k} (m_{\eps_k}, \bSig_{\eps_k}) \gt F(m, \bSig)$. Moreover, $(m, \bSig)$ is a minimizer of $F$.
\end{corollary}
Before we give the proof of \cref{thm:gamma}, let us first discuss the limit functional $F$ as well as its minimization. 
We assume that $m = x^{i_0}$ for some $i_0 \in \{ 1, \ldots, n\}$ and  rewrite the definition of $F_0(x^{i_0}, \bSig)$,  by adding and subtracting $\log (\beta^{i_0}) = -V_2(x^{i_0}) -\frac{1}{2} \log \left( \left(\det D^2V_1(x^{i_0})\right)\right)$ and cancelling the terms involving $V_2(x^{i_0})$ as
\be\label{eq:F0-1}
\bea
 F_0(x^{i_0}, \bSig) &= \frac{1}{2}\Tr\left(D^2V_1(x^{i_0})\cdot \bSig\right)- \frac{d}{2} - \frac{1}{2}\log\det(D^2V_1(x^{i_0})\cdot \bSig) \\
 & + \log \left(\sum_{i=1}^n \beta^i\right) - \log \left( \beta^{i_0}\right).
\ena
\en
Now it is interesting to see that the first line of \cref{eq:F0-1} gives the Kullback-Leibler divergence
$
D_{\text{KL}}\left( N(x^{i_0}, \bSig) \ ||\ N(x^{i_0}, (D^2 V_1(x^{i_0}))^{-1})\right).
$
The second line of \cref{eq:F0-1} is equal to the Kullback-Leibler divergence $D_{\text{KL}}(\bee^{i_0}\ ||\ \bbeta ),$ 
 for $\bee^{i_0}:=(0, \cdots, 1, \cdots, 0)$.
In conclusion, 
\be\label{eq:F0-2}
F_0(x^i, \bSig)  = D_{\text{KL}}\left( N(x^i, \bSig) \ ||\ N(x^i, (D^2 V_1(x^i))^{-1})\right) + D_{\text{KL}}(\bee^i\ ||\ \bbeta ),
\en
in other words, in the limit $\eps \lgt 0$, the Kullback-Leibler divergence between the best Gaussian measure $\nu_\eps$ and the measure $\mu_\eps$ consists of two parts: the first part is the relative entropy between the Gaussian measure with rescaled covariance $\bSig$ and the Gaussian measure with covariance determined by $(D^2 V_1(x^i))^{-1}$; the second part is the relative entropy between the Dirac mass supported at $x^i$ and a weighted sum of Dirac masses, with weights
$\bbeta$, at the $\{x^j\}_{j=1}^n$. Clearly, to minimize $F_0(m, \bSig)$, on the one hand, we need to choose $m = x^i$ and $\bSig = (D^2 V_1(x^i))^{-1}$ for some $i\in {1, \cdots, n}$; for this choice the first term on the right side of \cref{eq:F0-1} vanishes. In order to minimize the second term we need to choose the minimum $x^i$ 
with maximal weight $\beta^i$.
%
In particular, the following corollary holds. 

\begin{corollary}\label{cor:sgaussian}
The minimum of $F_0$ is zero when $n = 1$, but it  is strictly positive when $n > 1$. 
\end{corollary}
\cref{cor:sgaussian} reflects the fact that, in the limit $\eps \lgt 0$, a single Gaussian measure is not the best choice for approximating a non-Gaussian measure with
multiple modes; this motivates our study of Gaussian mixtures
in \cref{sec:gaussianmix}.

The proofs of \cref{thm:gamma} and \cref{cor:covmin} are provided after establishing a sequence of lemmas. 
The following lemma shows that the sequence of functionals $\{F_\eps\}$ is compact (recall Definition \cref{d:gcc}). It is well known that the Kullback-Leibler divergence (with respect to a fixed reference $\mu$) has compact sub-level sets with respect to weak convergence of probability measures. Here we prove a stronger statement, which is specific to the family of reference measures $\mu_\eps$, namely a uniform bound from above and below for the rescaled covariances, i.e. we prove a bound from above and below for $\bSig_\eps$ if we control $F_\eps(m_\eps, \bSig_\eps)$. 
 
\begin{lemma}\label{lem:compactness} Let $\{(m_\eps, \bSig_\eps)\} \subset \R^d \times \mathcal{S}_{\geq} (\R, d)$ be such that $\limsup_{\eps \lgt 0} F_\eps (m_\eps, \bSig_\eps) < \infty$. Then 
\be\label{eq:compactness}
0 < \liminf_{\eps \lgt 0}\lambda_{\min}(\bSig_\eps) < \limsup_{\eps \lgt 0}\Tr(\bSig_\eps) < \infty
\en
and $\text{dist}(m_\eps, \EE) \lgt 0$ as $\eps \lgt 0$. In particular, there exist common subsequences $\{m_k\}_{k\in \N}$ of $\{m_\eps\}$, $\{\bSig_k\}_{k\in \N}$ of $\{\bSig_\eps\}$ such that $m_{k} \gt x_{i_0}$ with $1 \leq i_0 \leq n$ and $\bSig_k \gt \bSig \in \mathcal{S}_{>}(\R, d)$.
\end{lemma} 

\begin{proof}
Let $ M := \limsup_{\eps \lgt 0} F_\eps (m_\eps, \bSig_\eps) < \infty$. Since $m_\eps$ and $\bSig_\eps$ are defined in finite dimensional spaces, we only need to show that both sequences are uniformly bounded. The proof consists of the following steps.

\textbf{Step 1}. We first prove the following rough bounds for $\Tr(\bSig_\eps)$: there exists positive constants $C_1, C_2$ such that when $\eps \ll 1$, 
\be\label{eq:bdsigma0}
C_1 \leq \Tr (\bSig_\eps) \leq \frac{C_2}{\eps}.
\en
In fact, from the formula \cref{eq:dkl_asym1} and the assumption that $V^\eps_1$ and $V_2$ are non-negative, we can get that when $\eps \ll 1$
\be\label{eq:logdetsig}
\log (\det \bSig_\eps)  \geq 2 (C_V - M - 1)
\en
 where the constant
$$
C_V := -\frac{d}{2} + \log \left(\sum_{i=1}^n \beta^i \right).
$$
Then the lower bound of \cref{eq:bdsigma0} follows from \cref{eq:logdetsig} and the arithmetic-geometric mean inequality 
\be\label{eq:trace-det}
\det \bA \leq \left(\frac{1}{d}\Tr(\bA)\right)^d
\en
 which holds for any positive definite $\bA$. 
 In addition, using the condition (A-5) for the potential $V^\eps_1$, we obtain from \cref{eq:dkl_asym1} that when $\eps \ll 1$,
\be\bea\label{eq:bdsigma1}
M & \geq F_\eps(m_\eps, \bA_\eps) \\
& \geq \E^{\nu_\eps} V_2(x)  + \frac{c_1}{\eps} \E^{\nu_\eps} |x|^2 -\frac{c_0}{\eps} - \frac{1}{2}\log\left(\det \bSig_\eps\right) + C_V - 1\\
& = \E^{\nu_\eps} V_2(x)  + c_1 \Tr(\bSig_\eps) + \frac{c_1 |m_\eps|^2}{\eps} -\frac{c_0}{\eps} - \frac{1}{2}\log\left(\det \bSig_\eps\right) + C_V - 1\\
& \geq c_1\Tr(\bSig_\eps) -\frac{c_0}{\eps}   - \frac{1}{2}\log \left(\left(\frac{1}{d} \Tr(\bSig_\eps)\right)^d\right) + C_V - 1\\
& = c_1 \Tr(\bSig_\eps) - \frac{c_0}{\eps}  - \frac{d}{2} \log(\Tr(\bSig_\eps)) + \frac{d\log d}{2} + C_V - 1 ,
\ena
\en
where we have used the inequality \cref{eq:trace-det} and the assumption that $V_2$ is non-negative. Dropping the non-negative terms on the right hand side we rewrite this expression as an estimate on $\Tr(\bSig_\eps)$, 
$$
c_1 \Tr(\bSig_\eps)  - \frac{d}{2} \log(\Tr(\bSig_\eps)) \leq M + \frac{c_0}{\eps} +1,
$$
and conclude that  there exists $C_2 > 0$ such that $\Tr(\bSig_\eps) \leq C_2/\eps$ by observing that for $x \gg 1$ we have $c_{1} x - \frac{d}{2} \log x \geq \frac{c_1}{2} x$.

\textbf{Step 2.} 
In this step we show that for $\eps \ll 1$ the mass of $\nu_\eps$ concentrates near the minimizers.
More precisely, we claim that there exist constants $R_1, R_2 > 0$, such that for every $\eps \ll 1$ there exists an index $i_0 \in \{1,2,\cdots, n\}$ such that 
\be\label{eq:bsigma2}
\nu_\eps\left( B\left(x_{i_0}, \sqrt{\eps(R_1 + R_2\log\left(\det \bSig_\eps\right))}\right)\right) \geq \frac{1}{2n}.
\en
On the one hand, from the expression \cref{eq:dkl_asym1} and the assumption that
$\limsup_{\eps \lgt 0} F_\eps (m_\eps, \bSig_\eps) \leq M$ we know that there exist $C_3, C_4 > 0$ such that when  $\eps \ll 1$
\be\label{eq:bsigma3}
\E^{\nu_\eps} V^\eps_1(x) \leq \eps\left(C_3 + C_4 \log\left(\det \bSig_\eps\right)\right).
\en 
On the other hand, it follows from \cref{eq:distV} that for $\eta \ll 1$
\be\bea
\E^{\nu_\eps} V^\eps_1(x) & \geq \E^{\nu_\eps} \left[V^\eps_1(x) \bI_{(\cup_{i=1}^n B(x_i, \eta))^c }(x)\right]\\
& \geq C_{\delta}\eta^2  \nu_\eps (\cup_{i=1}^n B(x_i, \eta))^c,
\ena
\en
which combined with \cref{eq:bsigma3} leads  to 
\be\label{jji}
\nu_\eps (\cup_{i=1}^n B(x_i, \eta))^c ) \leq \eps \frac{\left(C_3 + C_4 \log\left(\det \bSig_\eps\right)\right)}{C_{\delta}\eta^2 }.
\en
Now we choose $\eta = \eta_\eps := \sqrt{2 \eps(C_3 + C_4 \log (  \det \bSig_\eps))/ C_\delta}$ (by the rough bound \cref{eq:bdsigma0} this $\eta_\eps$ tends to zero as $\eps \to 0$, which permits to apply \cref{eq:distV}). 
This implies \cref{eq:bsigma2} with $R_1 = \frac{2C_3}{C_\delta}$ and $R_2 = \frac{2C_4}{C_\delta}$, 
by  passing to the complement and observing that 
$$
\sup_{i \in \{ 1, \ldots, n\}} \nu_\eps (B(x_i, \eta_\eps)) \geq \frac{1}{n} \nu_\eps \left( \cup_{i \in \{ 1, \ldots, n\}}B(x_i, \eta_\eps) \right).
$$

\textbf{Step 3.} We prove the bounds \cref{eq:compactness}. 
As in the previous step we set 
$$\eta_\eps = \sqrt{\eps(R_1 + R_2 \log(\det \bSig_\eps))}.
$$
 It follows from \cref{eq:bsigma2} that  
\be\bea\label{eq:bsigma4}
 \frac{1}{2n}& \leq \nu_\eps(B(x_{i_0}, \eta_\eps)) \\
& = \frac{1}{\sqrt{(2\pi\eps)^d \det \bSig_\eps}} \int_{B(x_{i_0}, \eta_\eps)} \exp\left(-\frac{1}{2\eps} \langle x - m_\eps, \bSig_\eps^{-1} (x - m_\eps) \rangle\right) dx\\
& \leq   \frac{1}{\sqrt{(2\pi\eps)^d \det \bSig_\eps}} | B(x_{i_0}, \eta_\eps)|\\
&   \leq C  \frac{1}{\sqrt{\eps^d \det \bSig_\eps}} \eta_\eps^d   \leq C \sqrt{  \frac{(R_1 + R_2 \log (\det \bSig_\eps))^d}{\det \bSig_\eps} }.
\ena
\en
This implies that $\limsup_{\eps \lgt 0}\det \bSig_\eps < C$ for some $C > 0$. In order to get a lower bound on 
individual eigenvalues $\bLam_\eps^{(i)}$ of $\bSig_\eps$, we rewrite the same integral in a slightly different way. 
We use the change of coordinates $y =  \frac{\mathbf{P}_\eps^T (x-m_\eps)}{\sqrt \eps}$, where $ \mathbf{P}_\eps$ is orthogonal and diagonalises $\bSig_\eps$ and observe that under this transformation
 $B( x^i, \eta_\eps)$ is mapped into $B(\frac{x^i-m}{\sqrt{\eps}}, \frac{\eta_\eps}{\eps}) \subseteq \{ y \colon |y_j - \frac{(x^{i}-m)}{\sqrt{\eps}}| \leq \frac{\eta_\eps}{\sqrt{\eps}} \quad \text{for } j=1, \ldots, n \}$. This yields 
\be\bea\label{eq:logdet}
\frac{1}{2n} & \leq \frac{1}{\sqrt{(2\pi)^d \det \bSig_\eps}} 
\int_{
\{|y_j  - \frac{(x^{i}-m)}{\sqrt{\eps}}| \leq \frac{\eta_\eps}{\sqrt{\eps}}\} }
 \exp\left(-\frac{1}{2}\langle y_i , (\bLam_{\eps}^{(i)})^{-1} y_i  \rangle\right) d y \\
& \leq  \frac{1}{\sqrt{(2\pi)^d \det \bSig_\eps}} \left(\frac{2\eta _\eps}{\sqrt{\eps}}\right)^{d-1} \int_{\R} \exp\left(-\frac{|y_i|^2}{2\bLam_\eps^{(i)}}\right) d y_i \\
& = \sqrt{\frac{\bLam_\eps^{(i)}}{(2\pi)^d \det \bSig_\eps}} \left(R_1 + R_2\log(\det \bSig_\eps)\right)^{\frac{d-1}{2}},
\ena
\en
for any $i \in \{1,2,\cdots, d\}$. Together with uniform boundedness of $\det \bSig_\eps$ this implies that $\bLam_\eps^{(i)} > C^\prime$ for some $C^\prime > 0$. Finally, 
\be\label{eq:upptracedet}
\Tr(\bSig_\eps) 
= \sum_{i=1}^d \bLam_\eps^{(i)} 
= \sum_{i=1}^d \frac{ \det (\bSig_\eps) }{\prod_{j=1, j\neq i}^d  \bLam_\eps^{(j)}}
\leq \frac{dC}{(C^\prime)^{d-1}} < \infty.
\en
This proves \cref{eq:compactness}.

\textbf{Step 4.} We show that $\text{dist}(m_\eps, \EE) \lgt 0$ as $\eps  \lgt 0$. 
On the one hand, by the upper bound on the variance in \cref{eq:compactness} and standard Gaussian concentration, we see that there exists a constant $c > 0$, such that for $\eps \ll 1$ we have $\nu_\eps(B(m_\eps, \sqrt{ \eps} c)) \geq \frac34$. On the other hand, we had already seen in \cref{jji} that for $\eta = \eta_\eps$ we have 
$$
\nu_\eps (\cup_{i=1}^n B(x_i, \eta_\eps))^c ) \leq \frac12,
$$
and hence $B(m_\eps, \sqrt{ \eps} c)$ must intersect at least one of the $B(x_i, \eta_\eps)$. This yields for this particular index $i$
$$
|x_i - m_\eps| \leq \eta_\eps + \sqrt{\eps} c,
$$
and establishes the claim.

\end{proof}
\begin{lemma}\label{lem:dkl-asym2}
Let $\{(m_\eps, \bSig_\eps)\}$ be a sequence such that $\limsup_{\eps \lgt 0} |m_\eps| =: C_1 <  \infty$ and
$$
0 < c_2 := \liminf_{\eps \lgt 0} \lambda_{\min} (\bSig_\eps) < \limsup_{\eps \lgt 0} \Tr(\bSig_\eps) =: C_2 < \infty.
$$
Then as $\eps \lgt 0$,
\be\bea\label{eq:dkl-asym2}
F_\eps(m_\eps, \bSig_\eps) & = \frac{V^\eps_1(m_\eps)}{\eps} + V_2(m_\eps) + \frac{1}{2}\Tr(D^2 V_1^\eps(m_\eps) \cdot \bSig_\eps) - \frac{1}{2}\log\left((2\pi\eps)^d\det \bSig_\eps\right)\\
& - \frac{d}{2} + \log Z_{\mu, \eps} + r_\eps
\ena\en
where $|r_\eps| \leq C\eps$ with $C = C(C_1, c_2, C_2, M_V)$ (Recall that  $M_V$ is the constant defined in \cref{assump} (A-1)).
\end{lemma}
\begin{proof}
The lemma follows directly from the expression \cref{eq:dkl} and Taylor expansion. Indeed, we first expand $V_2$ near $m_\eps$ up to the first order and then take expectation to get
$$
\E^{\nu_\eps} V_2(x) = V_2(m_\eps) + \E^{\nu_\eps} R_\eps(x)
$$
with residual 
$$
R_\eps(x) = \sum_{|\alpha|=2} \frac{(x - m_\eps)^\alpha}{\alpha!}\int_0^1 \partial^\alpha V_2\left(\xi x + (1 - \xi)m_\eps\right) (1 - \xi)^2 d \xi.
$$
Thanks to the condition (A-1), one can obtain the bound
\be\bea\label{eq:res0}
\E^{\nu_\eps} R_\eps(x) & \leq  \sum_{|\alpha| = 2} \frac{1}{\alpha!} \max_{\xi \in [0,1]} \left\{\E^{\nu_\eps} \left[|x - m_\eps|^2 \partial^\alpha V_2\left(\xi x + (1 - \xi)m_\eps\right)\right]\right\} \\
& \leq \frac{M_V}{\sqrt{(2\pi\eps)^d \det\mathbf{\Sigma_\eps}}} \max_{\xi \in [0,1]} \left\{\int_{\R^d} |x|^2 e^{(|x| + |m_\eps|)^2}\cdot e^{-\frac{1}{2\eps} x^T\mathbf{\Sigma}_\eps^{-1}x} dx \right\}\\
& \leq \frac{M_V}{\sqrt{(2\pi\eps)^d \det\mathbf{\Sigma}_\eps}} e^{2|m_\eps|^2}  \int_{\R^d}  |x|^2  e^{-\frac{1}{2\eps} x^T  (\mathbf{\Sigma}_\eps^{-1} - 4\eps \cdot \mathbf{I}_d) x} dx \\
& = \frac{M_V\eps}{\sqrt{\det\mathbf{\Sigma}_\eps}} e^{2|m_\eps|^2}  \cdot \det (\mathbf{\Sigma}_\eps^{-1} - 4\eps \cdot \mathbf{I}_d)^{-1}\\
& \leq C \eps,
\ena
\en
when $\eps \ll 1$. Note that in the last inequality we have used the assumption that all eigenvalues of $\bSig_\eps$ are bounded from above
which ensures 
that for $\eps \ll 1$ the matrix $\mathbf{\Sigma}_\eps^{-1} - 4\eps \cdot \mathbf{I}_d$ is positive definite.
 Hence 
 $$
 \E^{\nu_\eps} V_2(x) = V_2(m_\eps) + r_{1,\eps}
 $$ 
 with $r_{1,\eps} \leq C\eps$ as $\eps \lgt 0$. Similarly, one can take the fourth order Taylor expansion for $V_1^\eps$ near $m_\eps$ and then take expectation to obtain that 
$$
\E^{\nu_\eps} V_1^\eps(x) = \frac{V_1^\eps(m_\eps)}{\eps} + \frac{1}{2} \Tr\left(D^2 V_1^\eps(m_\eps) \cdot \bSig_\eps\right) + r_{2,\eps}
$$
with $r_{2,\eps} \leq C\eps$.
Then \cref{eq:dkl-asym2} follows directly by inserting the above equations into the expression \cref{eq:dkl}. 
\end{proof}

 The following corollary is a direct consequence of \cref{lem:normconst}, \cref{lem:compactness} and \cref{lem:dkl-asym2}, providing an asymptotic formula for $F_\eps(m_\eps, \bSig_\eps)$ as $\eps \lgt 0$. 
\begin{corollary}
Let $\{(m_\eps, \bSig_\eps)\} \subset \R^d \times \mathcal{S}_{\geq} (\R, d)$ be such that $\limsup_{\eps \lgt 0} F_\eps (m_\eps, \bSig_\eps) < \infty$. Then
for $\eps \ll 1$ 
\be\bea\label{eq:dkl-asym22}
F_\eps(m_\eps, \bSig_\eps) & = \frac{V^\eps_1(m_\eps)}{\eps} + V_2(m_\eps) + \frac{1}{2}\Tr(D^2 V(m_\eps) \cdot \bSig_\eps) - \frac{1}{2}\log\left(\det \bSig_\eps\right)\\
& - \frac{d}{2} + \sum_{i=1}^n \beta^i +  o(1).
\ena
\en
\end{corollary}

\begin{remark}\label{rem:3-9}
We do not have a bound on the convergence rate for  the residual  expression \cref{eq:dkl-asym22}, 
because \cref{lem:normconst} does not provide a convergence rate on the $Z_{\mu,\eps}$. This 
is because we do not impose any rate of convergence for the convergence of the $x^i_\eps$ to $x^i$. 
  The bound  $|r_\eps| \leq C\eps$ in 
\cref{lem:dkl-asym2} will be used to prove the rate of convergence for the posterior measures that arise from Bayesian inverse problems; see \cref{thm:expdkl} in \cref{sec:app}, and its proof.
\end{remark}

\begin{proof}[Proof of \cref{thm:gamma}]
 We first prove the liminf inequality. 
Let $(m_\eps, \bSig_\eps)$ be such that $m_\eps \gt m$ and $\bSig_\eps \gt \bSig$. We want to show that $F(m, \bSig) \leq \liminf_{\eps \lgt 0} F_\eps(m_\eps, \bSig_\eps)$. We may assume that $ \liminf_{\eps \lgt 0} F_\eps(m_\eps, \bSig_\eps) < \infty$ since otherwise there is nothing to prove. By Lemma~\cref{lem:compactness} this implies that $m\in \EE$ and that $\bSig$ is positive definite. 
Then the liminf inequality follows from \cref{eq:dkl-asym22} and the fact that $V_1^\eps \geq 0$. 

Next we show the limsup inequality is true. Given $m\in \EE, \bSig \in \mathcal{S}_{>}(\R, d)$,
we want to find recovery sequences $(m_k, \bSig_k)$ such that $(m_k,\bSig_k) \gt (m, \bSig)$ and $\limsup_{k} F_{\eps_k}(m_k, \bSig_k) \leq F(m, \bSig)$. In fact, we set $\bSig_k = \bSig$. Moreover, by \cref{assump} (A-4), we can choose $\{m_k\}$ to be one of the zeros of $V_1^{\eps_k}$ so that $V_1^{\eps_k} (m_k) = 0$ and $m_k \gt m\in \EE$. This implies that $V_2(m_k)\gt V_2(m)$. Then the limsup inequality follows from \cref{eq:dkl-asym22}.
\end{proof}

\begin{proof}[Proof of \cref{cor:covmin}]
First we show that $\limsup_{\eps \lgt 0} F_\eps(m_\eps, \bSig_\eps) < \infty$. In fact, let $\tilde{m}_\eps = x^1_\eps$ and $\tilde{\bSig}_\eps = D^2 V^\eps_1 (x^1_\eps)$. It follows from \cref{eq:dkl-asym22} that $\limsup_{\eps \lgt 0} F_\eps(\tilde{m}_\eps, \tilde{\bSig}_\eps) < \infty$. 
According to \cref{thm:fgamma}, the convergence of minima and minimizers is a direct consequence of \cref{lem:compactness} and \cref{thm:gamma}.
\end{proof}

\section{Approximation by Gaussian mixtures}\label{sec:gaussianmix}

In the previous section we demonstrated the approximation of
the target measure \cref{eq:mu_eps} by a Gaussian. Corollary 
\cref{cor:sgaussian} shows that, when the measure has only one
mode, this approximation is perfect in the limit $\eps \to 0$: the limit KL-divergence tends to zero since both entropies in \cref{eq:F0-2} tend to zero. However when multiple modes exist,
and persist in the small $\eps$ limit, the single Gaussian is inadequate because the relative entropy term $D_{\text{KL}}(\bee^{i} || \bbeta)$ can not be small even though the relative entropy between Gaussians tends to zero.
In this section we consider the approximation of the 
target measure $\mu_\eps$ by Gaussian mixtures in order to overcome this issue. 
We show that in the case of $n$ minimizers of $V_1$, the approximation 
with a mixture of $n$ Gaussians is again perfect as $\eps \to 0$.
The Gaussian mixture model is widely used in the pattern recognition and 
machine learning community; see the relevant discussion in \cite[Chapter 9]{B06}.

Let $\Delta^n$ be the standard $n$-simplex, i.e., 
$$
\Delta^n = \left\{\balpha = (\alpha^1, \alpha^2, \cdots, \alpha^n)\in \R^n: \alpha^i \geq 0 \text{ and } \sum_{i=1}^n \alpha^i = 1\right\}.
$$
For $\xi \in (0,1)$, we define
$
\Delta^n_{\xi} = \{\balpha = (\alpha^1, \alpha^2, \cdots, \alpha^n)\in \R^n: \alpha^i \geq \xi\}.
$

Recall that $\mathcal{A}$ is the set of Gaussian measures and define the set of Gaussian mixtures
\be\label{eq:Gaussmix}
\M_n = \left\{\nu = \sum_{i=1}^n \alpha^i \nu^i:  \nu^i  \in \A,\ \balpha = (\alpha^1, \alpha^2, \cdots, \alpha^n) \in \Delta^n \right\}.
\en
Also, for a fixed $\bxi = (\xi_1, \xi_2) \in (0,1)\times (0,\infty)$ we define the set 
\be\label{eq:Gaussmix-delta}
\bea
\M_n^{\bxi} = \Big\{\nu & = \sum_{i=1}^n \alpha^i \nu^i: \nu^i = N(m^i, \bSig^i)  \in \A \text{ with } \min_{i\neq j}|m^i - m^j| \geq \xi_2,\\
 & \balpha = (\alpha^1, \alpha^2, \cdots, \alpha^n) \in \Delta^n_{\xi_1} \Big\}.
 \ena
\en
While $\M_n$ is the set of all convex combinations of $n$ Gaussians taken from $\A$; the set $\M_n^{\bxi}$ can be seen as an ``effective'' version of $\M_n$, in which each Gaussian component plays an active role, and no two Gaussians share a common center.


Consider the problem of minimizing $D_{\text{KL}}(\nu || \mu_\eps)$ within $\M_n$ or $\M^{\bxi}_n$. 
Since the sets $\M_n$ and $\M^{\bxi}_n$ are both closed with respect to weak convergence, we have the following existence result whose proof is similar to \cref{thm:exists1} and is omitted.  
\begin{theorem}\label{thm:exists2}
Consider the measure $\mu_\eps$ given by \cref{eq:mu_eps} with fixed 
$\eps > 0$, and the problem of minimizing the functional
\be\label{eq:minfunc}
\nu \mapsto D_{\text{KL}}(\nu || \mu_\eps)
\en
from the set $\M_n$, or from the set $\M_n^{\bxi}$ with some fixed $\bxi = (\xi_1, \xi_2) \in (0, 1) \times (0, \infty)$. In both cases, there exists at least one minimizer to the functional \cref{eq:minfunc}. 
\end{theorem}

Now we continue to investigate the asymptotic behavior of the Kullback-Leibler approximations based on Gaussian mixtures. To that end, we again parametrize  a measure $\nu$ in the set $ \M_n$ or $ \M^\xi_n$ by the weights $\balpha = (\alpha^1,\alpha^2, \cdots, \alpha^n)$ as well as the $n$ means as well as the $n$ covariances matrices. 
 Similar to the previous section we need to chose the right scaling in our Gaussian mixtures to reflect the typical size of fluctuations of $\mu_\eps$. Thus for $\bm = (m^1, m^2, \cdots, m^n)$  and $\bSig = (\bSig^1, \bSig^2, \cdots, \bSig^n)$. 
we set
\be\label{eq:nu-form}
\nu_\eps = \sum_{i=1}^n \alpha^i N(m^i, \eps \bSig^i).
\en
We can view $D_{\text{KL}}(\nu_\eps || \mu_\eps)$ as a functional of $(\balpha, \bm, \bSig)$ and study the $\Gamma$-convergence of the resulting functional.   
 For that purpose, we need to restrict our attention to finding the best Gaussian mixtures within $\M^{\bxi}_n$ for some $\bxi\in (0,1)\times (0, \infty)$. The reasons are the following. First, we require individual Gaussian measures $\nu^i$ to be active (i.e. $\alpha^i > \xi_1 > 0$) because $D_{\text{KL}}(\nu_\eps, \mu_\eps)$, as a family of functionals of $(\balpha, \bm, \bSig)$ indexed by $\eps$, is not compact if we allow some of the $\alpha^i$ to vanish. In fact, if $\alpha_\eps^{i} = 0$ for some $i \in 1,2,\cdots, n$, then $D_{\text{KL}}(\nu_\eps || \mu_\eps)$ is independent of $m^{i}_\eps$ and $\bSig^{i}_\eps$. In particular, if $|m_\eps^{i} | \wedge |\bSig_\eps^{i}| \gt \infty$ while $|m_\eps^{j} | \vee |\bSig_\eps^{j}| < \infty$ for all the $j$'s such that $j \neq i$, then it still holds that $\limsup_{\eps \lgt 0} D_{\text{KL}}(\nu_\eps || \mu_\eps) < \infty$. Second, it makes more sense to assume that the individual Gaussian means stay apart from each other (i.e. $\min_{i\neq j}|m^i - m^j| \geq \xi_2 > 0$) since we primarily want to locate different modes of the target measure. Moreover, it seems impossible to identify a sensible $\Gamma$-limit without such an assumption; see \cref{rem:assum}.

Recall that the measure $\nu$ has the form \cref{eq:nu-form}. 
Let $\bxi=(\xi_1,\xi_2) \in (0,1)\times (0,\infty)$ be fixed. 
In view of these considerations it is useful to define
$$S_{\bxi}=\{(\balpha, \bm) \in \Delta^n_{\xi_1} \times \R^{nd}: \min_{i\neq j} |m^i - m^j| \geq \xi_2.\}$$

We define the functional 
\be\label{eq:dkl-2}
G_\eps(\balpha, \bm, \bSig) := \begin{cases}
D_{\text{KL}}(\nu || \mu_\eps) & \text{ if } (\balpha,\bm) \in S_{\bxi},\\ 
+\infty & \text{ otherwise}.
\end{cases}
\en
By the definition of the Kullback-Leibler divergence, if $(\balpha,\bm) \in S_{\bxi}$, then
\be\label{eq:dkl-3}
G_\eps(\balpha, \bm, \bSig) = \int \rho(x) \log \rho(x)dx + \frac{1}{\eps} \E^{\nu} V^\eps_1(x) + \E^\nu V_2(x) + \log Z_{\mu, \eps}
\en
where $\rho$ is the probability density function (p.d.f) of $\nu$.

Recall the $\Gamma$-limit $F$ defined in \cref{eq:F}. Then we have the following $\Gamma$-convergence result.
\begin{theorem}\label{thm:gamma-2}
The $\Gamma$-limit of $G_\eps$ is 
\be\bea\label{eq:limit-G}
 G(\balpha, \bm, \bSig)  & := 
\sum_{i=1}^n \alpha^i D_{\text{KL}}\left( N(m^i, \bSig^i) \ ||\ N(m^i, (D^2 V_1(m^i))^{-1})\right) 
\\
&\qquad + D_{\text{KL}}( \balpha\ ||\ \bbeta ) 
\ena
\en
if $(\balpha,\bm) \in S_{\bxi}$ and $m^i\in \EE$, 
 and $\infty$   otherwise.
\end{theorem}

\begin{remark}
The right hand side of $G$ consists of two parts: the first part is a weighted relative entropy which measures the discrepancy between two Gaussians, and the second part is the relative entropy between sums of Dirac masses
at $\{x^j\}_{j=1}^n$ with weights $\balpha$ and $\bbeta$ respectively. This has the same spirit as the entropy splitting used in \cite[Lemma 2.4]{MS14}.
\end{remark}

Before we prove \cref{thm:gamma-2}, we consider the minimization of the limit functional $G$. First let $\xi_1, \xi_2$ be such that 
\be\label{eq:xi}
 \xi_1 > 0, \quad 0 < \xi_2 \leq \min_{i\neq j} |x^i - x^j|,
\en
 where $\{x^i\}_{i=1}^n$ are the minimizers of $V_1$. To minimize $G$, without loss of generality, we may choose $m^i = \overline{m}^i := x^i$. Then the weighted relative entropy in the first term  in the definition \cref{eq:limit-G} of $G$ vanishes if we set $\bSig^i = \overline{\bSig}^i := D^2 V_1(x^i)^{-1}$. The relative entropy of the weights also vanishes if we choose the weight $\balpha = \overline{\balpha} := \bbeta$. To summarize, the minimizer $(\overline{\balpha}, \overline{\bm}, \overline{\bSig})$ of $G$ is given by
\be\label{eq:minimizermix}
\overline{m}^i = x^i,\ \overline{\bSig}^i = D^2 V_1(x^i)^{-1}, \overline{\alpha}^i = \beta^i,
\en
and $G(\overline{\balpha}, \overline{\bm}, \overline{\bSig}) = 0$. 
The following corollary is a direct consequence of the $\Gamma$-convergence of $G_\eps$.
\begin{corollary}\label{cor:convmin}
Let $\{(\balpha_\eps, \bm_\eps, \bSig_\eps)\}$ be a family of
minimizers of $\{G_\eps\}$. Then there exists a subsequence $\{\eps_k\}$ such that $(\balpha_{\eps_k}, \bm_{\eps_k}, \bSig_{\eps_k}) \gt (\overline{\balpha}, \overline{\bm}, \overline{\bSig})$ and that $G_{\eps_k}(\balpha_{\eps_k}, \bm_{\eps_k}, \bSig_{\eps_k}) \gt G(\overline{\balpha}, \overline{\bm}, \overline{\bSig})$. Moreover, $(\overline{\balpha}, \overline{\bm}, \overline{\bSig})$ is a minimizer of $G$ and $G(\overline{\balpha}, \overline{\bm}, \overline{\bSig}) = 0$. 
\end{corollary}

For a non-Gaussian measure $\mu_\eps$ with multiple modes, i.e., $n > 1$ in the \cref{assump}, we have seen in \cref{cor:sgaussian} that the Kullback-Leibler divergence between $\mu_\eps$ and the best Gaussian measure selected from $\A$ remains positive as $\eps \lgt 0$. However, this gap is filled by using Gaussian mixtures, namely, with $\nu_\eps$ being chosen as the best Gaussian mixture, the Kullback-Leibler divergence $D_{\text{KL}}(\nu_\eps || \mu_\eps) \lgt 0$ as $\eps \lgt 0$. 


Similarly to the proof of \cref{thm:gamma}, \cref{thm:gamma-2} 
follows directly from \cref{cor:dkl-asym4} below, whose proof requires several lemmas.  We start by showing the compactness of $\{G_\eps\}$.
\begin{lemma}\label{lem:compactness2}
Let $G_\eps$ be defined by \cref{eq:dkl-2}. Fix $\bxi  = (\xi_1, \xi_2)$ satisfying the condition \eqref{eq:xi}. 
Let $\{(\balpha_\eps, \bm_\eps, \bSig_\eps)\}$ be a sequence in $S_{\bxi} \times \mathcal{S}_{\geq }(\R, d)$ such that 
$$\limsup_{\eps \lgt 0} G_\eps(\balpha_\eps, \bm_\eps, \bSig_\eps) < \infty.$$ Then 
\be\label{eq:compactness2}
\liminf_{\eps\lgt 0} \min_i \lambda_{\min}(\bSig^i_\eps) > 0, \ \limsup_{\eps \lgt 0} \max_i   \Tr(\bSig^i_\eps) < \infty
\en
and $\text{dist}(m^i_\eps, \EE) \lgt 0$ as $\eps \lgt 0$. In particular, for any $i$, there exists $ j = j(i)\in \{1,2,\cdots, n\}$ and a subsequence $\{m^i_k\}_{k\in \N}$ of $\{m^i_\eps\}$ such that $m_{k} \gt x_{j}$ as $k \gt\infty$.
\end{lemma}
\begin{proof}
We write $M = \limsup_{\eps \lgt 0} G_\eps(\balpha_\eps, \bm_\eps, \bSig_\eps)$ and 
$$
\nu_\eps = \sum_{i=1}^n \alpha_\eps^i \nu^i_\eps
$$
where $\nu^i_\eps = N(m^i_\eps, \eps \bSig^i_\eps)$. Then we get
\begin{align*}
D_{KL}(\nu_\eps|| \mu_\eps)& = \sum_{j=1}^n \alpha^j_\eps \;\E^{\nu^j_\eps} \log \left( \sum_{i} \alpha^i_\eps \frac{ d\nu^i_\eps}{d\mu_\eps}\right)\\
& \geq \sum_{j=1}^n \alpha^j_\eps  \; \E^{\nu^j_\eps} \log \left( \alpha^j_\eps \; \frac{d \nu^j_\eps}{d\mu_\eps}\right)\\
& = \sum_{j=1}^n \alpha^j_\eps\log(\alpha^j_\eps) + \sum_{j=1}^n  \alpha^j_\eps  \;\E^{\nu^j_\eps} \log \left( \frac{d \nu^j_\eps}{d\mu_\eps}\right)\\
& = \sum_{j=1}^n \alpha^j_\eps\log(\alpha^j_\eps) + \sum_{j=1}^n \alpha^j_\eps D_{KL}(\nu^j_\eps || \mu_\eps )
\end{align*}
where the inequality follows simply from the monotonicity of the logarithmic function. 
As each of term $D_{KL}(\nu^j_\eps || \mu_\eps)$ is non-negative, this implies the bound
$$
 D_{KL}(\nu^j_\eps || \mu_\eps ) \leq \frac{1}{\alpha^j_\eps} \left(M - n \min_{\alpha \in [0,1]} \alpha \log \alpha \right).
$$ 
Using the lower bound $\alpha^j_\eps > \xi_1$ which holds by assumption we get a uniform upper bound on $D_{KL}(\nu^j_\eps || \mu_\eps )$ which in turn permits to invoke Lemma~\cref{lem:compactness}.

\end{proof}

\begin{lemma}\label{lem:dkl-asym3}
Let $\{(\balpha_\eps, \bm_\eps, \bSig_\eps)\}$ be a sequence in $S_{\bxi} \times S_{\geq}(\R, d)$ with $\bxi = (\xi_1, \xi_2)$ satisfying \eqref{eq:xi} and such that  
$$
c_1 \leq \liminf_{\eps\lgt 0} \min_i \lambda_{\min}(\bSig^i_\eps)  < \limsup_{\eps \lgt 0} \max_i |m^i_\eps| \vee \Tr(\bSig^i_\eps) \leq C_1 < \infty.
$$
 Then  
\be\bea\label{eq:dkl-asym3}
&G_\eps(\balpha_\eps, \bm_\eps, \bSig_\eps)\\
 & = \sum_{i=1}^n \alpha^i_\eps \left(\frac{V^\eps_1(m^i_\eps)}{\eps} + V_2(m^i_\eps)  + \frac{1}{2} \Tr(D^2 V^\eps_1(m^i_\eps) \cdot \bSig^i_\eps) - \frac{1}{2}\log\left(\det \bSig^i_\eps\right) \right)\\
&  + \sum_{i=1}^n \alpha^i_\eps \log \alpha^i_\eps  - \frac{d}{2}  + \log Z_{\mu, \eps} + r_\eps.
\ena\en
 where $r_\eps \leq C\eps$ with $C = C(c_1, C_1, M_V,\xi_2)$.
\end{lemma}

\begin{proof}
By assumption, we know from \cref{eq:dkl-3}  that
$$
G_\eps(\balpha_\eps, \bm_\eps, \bSig_\eps) = \int \rho_\eps(x) \log \rho_\eps(x)dx + \frac{1}{\eps} \E^{\nu_\eps} V^\eps_1(x) + \E^{\nu_\eps} V_2(x) + \log Z_{\mu, \eps}
$$
where $\rho_\eps = \sum_{i=1}^n \alpha^i_\eps \rho^i_\eps$ is the probability density of the measure $\nu_\eps$.
First of all, applying the same Taylor expansion arguments used 
to obtain \cref{eq:dkl-asym2}, one can deduce that
\be\bea\label{eq:expV}
&\frac{1}{\eps} \E^{\nu_\eps} V^\eps_1(x) + \E^{\nu_\eps} V_2(x) \\
&= \sum_{i=1}^n \alpha^i_\eps \left(\frac{V^\eps_1(m^i_\eps)}{\eps} + \frac{1}{2} \Tr\left(\nabla^2 V^\eps_1(m^i_\eps) \cdot \bSig^i_\eps\right) +  V_2(m^i_\eps) \right) + r_{1,\eps}
\ena\en
with $r_{1,\eps} \leq C\eps$ and $C = C(C_1, c_1, M_V)$. 
Next, we claim that the entropy of $\rho_\eps$ can be rewritten as 
\be\label{eq:rhologrho-0}
\int \rho_\eps(x) \log \rho_\eps(x)dx  = \sum_{i=1}^n \alpha^i_\eps \left(\int \rho^i_\eps(x) \log \rho^i_\eps(x)dx + \log \alpha^i_\eps\right) + r_{2, \eps}
\en
where $r_{2,\eps} \leq e^{-\frac{C}{\eps}}$ when $\eps \ll 1$ with the constant $C = C(C_1, c_2, \xi_2)$. By definition, 
$$
\int \rho_\eps(x) \log \rho_\eps(x)dx  = \sum_{i=1}^n \alpha^i_\eps \int 
\rho^i_\eps(x)  \log \left(\sum_{j=1}^n \alpha^j_\eps \rho^j_\eps(x)\right)dx,
$$
so it suffices to show that for each $i \in \{ 1, \ldots,n\}$ we have 
\be\label{eq:rhologrho-1}
\int \rho^i_\eps(x)  \log \left(\sum_{j=1}^n \alpha^j_\eps \rho^j_\eps(x)\right)dx = \int \rho^i_\eps(x) \log \rho^i_\eps(x)dx + \log \alpha^i_\eps + r_{2,\eps}
\en
with $r_{2, \eps} \leq e^{-\frac{C}{\eps}}$. Indeed, the monotonicity of the logarithmic function yields
\be\label{eq:rhologrho-2}
\int \rho^i_\eps(x)  \log \left(\sum_{j=1}^n \alpha^j_\eps \rho^j_\eps(x)\right) dx\geq \int \rho^i_\eps(x) \log \rho^i_\eps(x)dx + \log \alpha^i_\eps. 
\en
In order to show the matching lower bound we first recall that the means $m^i_\eps$ of the $\nu^i_\eps$ are well separated by assumption,  $\min_{j\neq i} |m^i_\eps - m^j_\eps| > \xi_2$. Let $\delta \ll \frac{\xi}{2}$ to be fixed below  and set $B^i_\delta =B(m^i_\eps, \delta)$  Then we write
\be\label{eq:rhologrho-3}
\bea
  \int  & \rho^i_\eps  \log \Big(\sum_{j=1}^n \alpha^j_\eps \rho^j_\eps\Big) \\
 &=  \int \rho^i_\eps \log\left(\alpha^i_\eps \rho^i_\eps\right) 
 +   \int_{B^i_\delta} \rho^i_\eps  \Big( \log\Big(\sum_{j=1}^n \alpha^j_\eps \rho^j_\eps\Big) - \log\Big(\alpha^i_\eps \rho^i_\eps\Big) \Big) \\
& \qquad  + \int_{(B^i_\delta)^c} \rho^i_\eps\Big( \log\Big(\sum_{j=1}^n \alpha^j_\eps \rho^j_\eps\Big) - \log\Big(\alpha^i_\eps \rho^i_\eps\Big)\Big) \\
&=:\Big( \int \rho^i_\eps \log  \rho^i_\eps  + \log \alpha^i_\eps  \Big) + E^1_\eps + E^2_\eps.
\ena
\en

We first show that the error term  $E^2_\eps$ is exponentially small. To that end, we first drop the exponential term in the Gaussian density to  obtain the crude bound 
\be\bea\label{eq:bd1}
\log \Big(\sum_{j=1}^n \alpha^j_\eps \rho^j_\eps\Big) 
\leq \log \Big(\sum_{j=1}^n \alpha^j_\eps  \frac{1}{\sqrt{(2\pi\eps)^d \det \bSig^j_\eps}} \Big)\leq \frac{d}{2} \log \eps^{-1} +C.
\ena
\en
where in the second inequality we use the fact that $\det \bSig^i_\eps$ is bounded away from zero, which has been established in \cref{eq:compactness2}. Moreover, by definition we have
\be\label{eq:bd2}
- \log \Big(\alpha^i_\eps \rho^i_\eps\Big) \leq \frac{d}{2} \log \eps ^{-1} + C + \frac{|x - m^i_\eps|^2}{\eps}
\en
 Plugging bounds \cref{eq:bd1} and \cref{eq:bd1} in  and using Gaussian concentration as well as the lower bound on $\lambda_{\min}$ established in \cref{lem:compactness2}
\be\bea\label{eq:rhologrho-4}
E^2_\eps & \leq
\int_{(B^i_\delta)^c} \rho^i_\eps(x)  
  \Big( \frac{d}{2} \log \eps^{-1} +C +\frac{|x - m^i_\eps|^2}{\eps} \Big)  dx
 \leq C \big( \log \eps^{-1}  + \eps^{-1} \big) e^{-\frac{C\delta}{\eps}}
\ena
\en
when $\eps \ll 1$. Next, we want to bound $E^1_\eps$. Notice that $m^j_\eps \gt m^j$ for $j = 1, \cdots, n$, hence if $x\in B_\delta^i$ and if $\delta < \xi_1$, then $|x - m^j_\eps| > \xi_1 - \delta $ for any $j \neq i$ when $\eps \ll 1$. As a consequence, 
\be
\int_{B^i_\delta} \sum_{j=1, j \neq i}^n \alpha^j_\eps \rho^j_\eps \leq C\eps^{-\frac{d}{2}} e^{-\frac{C (\xi_1 - \delta)^2}{\eps}}.
\en
This together with the elementary inequality 
$$\log(x + y) = \log (x) + \int_x^{x+y} \frac{1}{t} dt \leq \log x + \frac{y}{x}$$
 for $x, y > 0$ implies
\be\label{eq:rhologrho-5}
\bea
  E^1_\eps & = \int_{B^i_\delta} \rho^i_\eps\Big( \log\Big(\alpha^i_\eps \rho^i_\eps +  \sum_{j=1, j\neq i}^n \alpha^j_\eps \rho^j_\eps\Big) - \log\Big(\alpha^i_\eps \rho^i_\eps\Big) \Big)\\
&
  \leq \int_{B^i_\delta} \frac{\sum_{j=1, j \neq i}^n \alpha^j_\eps \rho^j_\eps}{ \alpha^i_\eps}\\
&  \leq C
\delta^d \eps^{-\frac{d}{2}} e^{-\frac{C (\xi_1 - \delta)^2}{\eps}}.
\ena
\en
where we used that $\alpha^i_\eps$ is bounded below from zero. 
Hence \cref{eq:rhologrho-1} follows directly from \cref{eq:rhologrho-2}-\cref{eq:rhologrho-5}. 
 
Finally, \cref{eq:dkl-asym3} follows from combining \cref{eq:expV}, \cref{eq:rhologrho-0} and the identity
$$ \int \rho_\eps^i(x) \log \rho_\eps^i(x)dx = -\frac{1}{2}\log \left((2\pi\eps)^d\det \bSig^i_\eps\right) - \frac{d}{2}.$$  
\end{proof}

\begin{remark}\label{rem:assum}
The assumption that $\min_{j\neq i} |m^i_\eps - m^j_\eps| > \xi_2 > 0$ is the crucial condition that allows us to express the entropy of the Gaussian mixture in terms of the mixture of entropies of individual Gaussian (i.e. the equation \cref{eq:rhologrho-0}), leading to the asymptotic formula \cref{eq:dkl-asym3}. Neither formula \cref{eq:rhologrho-0} nor \cref{eq:dkl-asym2} is likely to be true without such an assumption since the cross entropy terms are not 
negligible.
\end{remark}
The following corollary immediately follows from Lemma~\cref{lem:dkl-asym3} by plugging in the Laplace approximation of the normalization constant $Z_{\mu, \eps}$ given in Lemma~\cref{lem:normconst} and rearranging the terms.

\begin{corollary}\label{cor:dkl-asym4}
Suppose that $\{(\balpha_\eps, \bm_\eps, \bSig_\eps)\}$ satisfy the same assmption as in \cref{lem:compactness2}. If $\limsup_{\eps \lgt 0} G_\eps(\balpha_\eps, \bm_\eps, \bSig_\eps) < \infty$, then  
\be\bea\label{eq:dkl-asym4}
&G_\eps(\balpha_\eps, \bm_\eps, \bSig_\eps) \\
& = 
\sum_{i=1}^n \alpha^i_\eps \left(\frac{V^\eps_1(m^i_\eps)}{\eps} + V_2(m^i_\eps)   -\frac{d}{2}  + \frac{1}{2} \Tr(D^2 V^\eps_1(m^i_\eps) \cdot \bSig^i_\eps) \right)\\
& + \sum_{i=1}^n \alpha^i_\eps \Big( \log \alpha^i_\eps - \frac{1}{2}\log\left(\det \bSig^i_\eps  \right)+  \log \Big(\sum_{j=1}^n \beta^j \Big)  \Big) + o(1).
\ena
\en
\end{corollary}
\begin{remark}\label{rem:4-9}
Similarly to the discussion in \cref{rem:3-9}, the residual in 
\cref{eq:dkl-asym4} is here demonstrated to be of order $o(1)$, 
but the quantitative bound that $|r_\eps| \leq C\eps$ in \cref{eq:dkl-asym3} 
can be used to extract a rate of convergence. This can be used
to study the limiting behaviour of posterior measures arising from 
Bayesian inverse problems when multiple modes are present; see the
next section.
\end{remark}

\section{Applications in Bayesian inverse problems}\label{sec:app}

 Consider the inverse problem of recovering $x\in \R^d$ from the noisy data $y\in \R^d$, where $y$ and $x$ are linked through the equation
\be
y = G(x) + \eta. 
\en 
Here $G$ is called the forward operator which maps from $\R^d$ into itself, $\eta \in \R^d$ represents the observational noise. 
We take a Bayesian approach to solving the inverse problem. The main idea is 
to first model our knowledge about $x$ with a prior probability 
distribution, leading to a joint distribution on $(x,y)$ once the probabilistic
structure on $\eta$ is defined. We then update the prior based on the observed 
data $y$; specifically we obtain the posterior distribution $\mu^y$ 
which is the conditional distribution of $x$ given $y$,
and is the solution to the Bayesian inverse problem. From this posterior measure
one can extract information about the unknown quantity of interest. We remark that since $G$ is non-linear in general, the posterior is generally not Gaussian even when the noise and prior are both assumed to be Gaussian. A systematic treatment of the Bayesian approach to inverse problems may be found in \cite{S10a}. 

In Bayesian statistics there is considerable interest in the study of 
the asymptotic performance of posterior measures from a frequentist perspective;
this is often formalized as the {\em posterior consistency}. To 
define this precisely, consider a sequence of observations $\{y_j\}_{j\in \N}$, generated from the truth $x^\dagger$ via
 \be\label{eq:yn}
y_j = G(x^\dagger) + \eta_j,
 \en
where $\{\eta_j\}_{j\in \N}$ is a sequence of random noises. This may model a statistical experiment with increasing amounts of data or with vanishing noise. In either case, posterior consistency refers to concentration of the posterior distribution around the truth as the data quality increases. For parametric statistical models, Doob's consistency theorem \cite[Theorem 10.10]{V00} guarantees posterior consistency under the identifiability assumption about the forward model. For nonparametric models, in which the parameters of interest lie in infinite dimensional spaces, the corresponding posterior consistency is a much more challenging problem. Schwartz's theorem \cite{S65,BSW99} provides one of the main theoretical tools to prove posterior consistency in infinite dimensional space, which replaces identifiability by a stronger assumption on testability. The posterior contraction rate, quantifying the speed that the posterior contracts to the truth, has been determined in various Bayesian statistical models (see \cite{GGV00,SW01,CN13}). In the context of the Bayesian inverse problem, the posterior consistency problem 
has mostly been studied to date for linear inverse problems with Gaussian priors \cite{KVZ11,ALS13}. The recent paper \cite{V13} studied posterior consistency for a 
specific nonlinear Bayesian inverse problem, using the stability estimate of the 
underlying inverse problem together with posterior consistency results for the Bayesian regression problem. 

In this section, our main interest is not in the consistency of posterior distribution, but in characterizing in detail its asymptotic behavior. We will consider two limit processes in \cref{eq:yn}: the small noise limit and the large data limit. In the former case, we assume that the noise $\eta_i = \frac{1}{\sqrt{i}}\eta$ where $\eta$ is distributed according to the standard normal $N(0, \bI_d)$, and we consider the data $\by_N$ given by the most accurate observation, i.e. $\by_N = y_N$. In the later case, the sequence $\{\eta_i\}_{i\in \N}$ is assumed to be independent identically distributed according to the standard normal and we accumulate
the observations so that the data $\by_N = \{y_1, y_2, \cdots, y_N\}$. In addition, 
 assume that the prior distribution is $\mu_0$ which has the density 
$$
\mu_0(dx) = \frac{1}{Z_0} e^{-V_0(x)} dx
$$
with the normalization constant $Z_0 > 0$. Since the data and the posterior are fully determined by the noise $\bdeta$ with $\bdeta = \eta$ or $\bdeta = \{\eta_i\}_{i\in \N}$, we denote the posterior by $\mu^{\bdeta}_N$ to indicate the dependence. By using Bayes's formula, we calculate the posterior distribution for both limiting cases below.

\begin{itemize}
 
\item Small noise limit

\be\bea\label{eq:postdist1}
\mu^{\bdeta}_N(dx) & = \frac{1}{Z_{N,1}^{\bdeta}} \exp\left(-\frac{N}{2}\left|y_n - G(x)\right|^2\right)\mu_0(dx)\\
& = \frac{1}{Z_{N,1}^{\bdeta}} \exp\left(-\frac{N}{2}\left|G(x^\dagger) - G(x) + \frac{1}{\sqrt{N}}\eta\right|^2\right)\mu_0(dx).
\ena
\en

\item Large data limit

\be\bea\label{eq:postdist2}
\mu^{\bdeta}_N(dx) & = \frac{1}{Z_{N,2}^{\bdeta}} \exp\left(-\frac{1}{2}\sum_{i=1}^N|y_i - G(x)|^2\right)\mu_0(dx)\\
& = \frac{1}{Z_{N,2}^{\bdeta}} \exp\left(-\frac{1}{2}\sum_{i=1}^N\left|G(x^\dagger) - G(x) + \eta_i\right|^2\right)\mu_0(dx).
\ena
\en

\end{itemize}

In both cases, we are interested in the limiting behavior of the posterior distribution $\mu^{\bdeta}_N$ as $N \gt \infty$. For doing so, we assume the forward operator $G$ satisfies one of the following assumptions. 

 \begin{assumption}\label{assum-bip}

(i) $G\in C^4(\R^d;\R^d)$ and $G(x) = G(x^\dagger)$ implies $x = x^\dagger$. Moreover, $G$ is a $C^1$-diffeomorphism in the neighborhood of $x^\dagger$.

(ii) $G\in C^4(\R^d;\R^d)$ and the zero set of the equation $G(x) = G(x^\dagger)$ is $\{x^\dagger_i\}_{i=1}^n$. Moreover
$x^\dagger_1 = x^\dagger$ and $G$ is a $C^1$-diffeomorphism  in the neighborhood of $x_i^\dagger$.

\end{assumption}

The following model problem gives a concrete example where these  assumptions are satsified. 

\subsection*{Model Problem}
Consider the following one dimensional elliptic problem
\be\label{eq:elliptic}
\begin{aligned}
& - u^{\prime\prime}(x) + \exp(q(x)) u(x)  = f(x), \quad x\in (0,1),\\
& u(0) = u(1) = 0.
\end{aligned}
\en
Here we assume that $q,f \in L^\infty(0,1)$ and that $f$ is positive on $(0,1)$. The inverse problem of interest is to find $q$ from the knowledge of the solution $u$. We restrict ourselves to a finite dimensional version of  \eqref{eq:elliptic}, which comes from the finite difference discretization
\be\label{eq:delliptic}
\begin{aligned}
& -\frac{u_{k+1} - 2u_{k} + u_{k-1}}{h^2} + e^{q_k} u_k = f_k, \quad k = 1,2,\cdots, M, \\
& u_0 = u_{M+1} = 0.
\end{aligned}
\en
Here $u_k$, $f_k$ and $q_k$ are approximations to $u(x_k)$,  $f(x_k)$ and $q(x_k)$) with $x_k = k/M, k = 1, \cdots, M$ and $h = 1/(M+1)$. The corresponding finite dimensional inverse problem becomes finding the vector $\mathbf{q} = \{q_k\}_{k=1}^M$ from the solution vector $\mathbf{u} = \{u_k\}_{k=1}^M$ given the right side $\mathbf{f} = \{f_k\}_{k=1}^M$. For ease of notation, let us denote by $\mathcal{A}$ the matrix representation  of the one dimensional discrete Laplacian, i.e. $\mathcal{A}_{ii} = 2/h^2$ for $i = 1, 2,\cdots, M$ and $\mathcal{A}_{ij} = -1/h^2$ when $|i-j|=1$. Let $\mathcal{Q}$ be the diagonal matrix with $\mathcal{Q}_{ii} = e^{q_i}, i=1,2,\cdots, M$.  With these notations, we can write the forward map $G$ as
$$
G: \mathbf{q} \in \R^M \gt \mathbf{u} \in \R^M \quad\quad  \mathbf{u} = G(\mathbf{q}) = (\mathcal{A} + \mathcal{Q})^{-1} \mathbf{f}.
$$
Note that both $\mathcal{A}$ and $\mathcal{Q}$ are positive definite so that $(\mathcal{A} + \mathcal{Q})$ is invertible.
We now discuss this forward map, and variants on it, in relation to
Assumption \ref{assum-bip}.

First consider Assumption \ref{assum-bip} (i). First, $G$ is smooth in $\mathbf{q}$ since $\mathcal{Q}$ depends smoothly on $\mathbf{q}$. In particular, for any fixed $\mathbf{q}\in \R^M$ with  corresponding solution vector $\mathbf{u}$, a direct calculation shows that the derivative matrix $D_\mathbf{q} G$ of the forward map $G$ is given by
$$
  D_\mathbf{q} G = (\mathcal{A} + \mathcal{Q})^{-1} \mathcal{U}\mathcal{Q}.
$$
Here $\mathcal{U}$ is a diagonal matrix with the diagonal vector $\mathbf{u}$. Due to our assumption that $f_k$ are positive, it follows from the (discrete) maximum principle that the $u_k$ are also positive, which in turn implies that $\mathcal{U}$ is invertible. Consequently, the matrix $D_\mathbf{q} G$ is invertible and 
$$
D_\mathbf{q} G^{-1} = \mathcal{Q}^{-1} \mathcal{U}^{-1} (\mathcal{A} + \mathcal{Q}).
$$
According to the inverse function theorem, the map $G:\R^M \gt \R^M$ is invertible at every $\mathbf{q}\in \R^M$ and its inverse $G^{-1}(\mathbf{u})$ is smooth in $\mathbf{u}$. Therefore Assumption \ref{assum-bip} (i) is fulfilled for any $x^\dagger = \mathbf{q}^\dagger \in \R^M$. The problem \eqref{eq:elliptic} can be modified slightly so that Assumption \ref{assum-bip} (ii) is satisfied. In fact,
consider the problem \eqref{eq:elliptic} with the coefficient $\exp(q)$ replaced by $q^2$. Then Assumption \ref{assum-bip} (ii) is satisfied for any $x^\dagger = \mathbf{q}^\dagger$ without zero entries. More specifically, the resulting forward map in this case is still smooth, but the equation $G(\mathbf{q}) = G(\mathbf{q}^\dagger)$ has $n = 2^M$ solutions $\{\mathbf{q}_i^\dagger\}_{i=1}^{2^M}$ corresponding to the fact that $\mathbf{q}$ is only determined up to a sign in each entry. Moreover, if $\mathbf{q}^\dagger$ has no zero entry, $G^{-1}$ is smooth near each of $\mathbf{q}_i^\dagger$.

We divide our exposition below according to whether the noise is fixed or is considered as a random variable.
For a fixed realization of noise $\bdeta = \eta$, by applying the theory developed in the previous section, we show the asymptotic normality for $\mu^{\bdeta}_N$ in the small noise limit. Furthermore, we obtain a Bernstein-Von Mises type theorem for $\mu^{\bdeta}_N$ with respect to both limit processes, small noise and large data.

\subsection{Asymptotic Normality}

In this subsection, we assume that the data is generated from the truth $x^\dagger$ and a single realization of the Gaussian noise $\eta^\dagger$, i.e.
$$
y = G(x^\dagger) + \frac{1}{\sqrt{N}} \eta^\dagger.
$$
Then the resulting posterior distribution $\mu^{\bdeta}_N$ has a density of the form
\be\label{eq:muepsBIP}
\bea
\mu^{\bdeta}_N(dx) & = \frac{1}{Z_{N}^{\bdeta}} \exp\left(-\frac{N}{2} |y - G(x)|^2 - V_0(x) \right) dx \\
& =  \frac{1}{Z_{N}^{\bdeta}} \exp\left(-\frac{N}{2} |G(x^\dagger)- G(x) + \frac{1}{\sqrt{N}} \eta^\dagger|^2 - V_0(x) \right) dx
\ena
\en
where $Z_{N}^{\bdeta}$ is the normalization constant. Notice that $\mu^{\bdeta}_N$ has the same form as the measure defined in \cref{eq:mu_eps} with $\eps = \frac{1}{N}, V^\eps_1 (x) = V_1^N(x) := \frac{1}{2}|G(x^\dagger) - G(x) + \frac{1}{\sqrt{N}} \eta^\dagger|^2 $ and $V_2(x) = V_0(x)$.


Suppose that $V_0\in C^2(\R^d; \R)$ and that $G$ satisfies one of the assumptions in Assumption \eqref{assum-bip}. Then the potentials $V^\eps_1$ and $V_2$ satisfy \cref{assump}. In particular, we have $V^\eps_1(x) \gt V_1(x) := \frac{1}{2}|G(x^\dagger) - G(x) |^2$ for any $x\in \R^d$ and that $D^2 V_1(x^\dagger_i) = DG(x^\dagger_i)^T DG(x^\dagger_i)$.
Recall the set of Gaussian measures $\A$  and the set of Gaussian mixtures $\M_n$ and $\M_n^{\bxi}$ (defined in \cref{eq:Gaussmix} and \cref{eq:Gaussmix-delta}). Again, we set $\bxi = (\xi_1, \xi_2)$ such that $\xi_1\in (0,1)$ and $\min_{i \neq j} |x^i - x^j| \geq \xi_2 > 0$. 
The following theorem concerning the asymptotic normality of $\mu_N^{\bdeta}$ is a direct consequence of \cref{cor:covmin} and \cref{cor:convmin}. 

\begin{theorem} \label{thm:bip}
\item [(i)] Let $V_0 \in C^2(\R^d; \R)$ and $G$ satisfy \cref{assum-bip} (i). Given any $N\in \N$, let $\nu_N = N(m_N, \frac{1}{N}\bSig_N) \in \A$ be a minimizer of the functional $\nu \mapsto D_{\text{KL}}(\nu || \mu_{N}^{\bdeta})$ within $\A$. Then $D_{\text{KL}}(\nu_N || \mu_N^{\bdeta}) \lgt 0$ as $N\gt \infty$. Moreover, $m_N \gt x^\dagger$ and $\bSig_N \gt \left(DG(x^\dagger)^T DG(x^\dagger)\right)^{-1}$.

\item [(ii)] Let $V_0 \in C^2(\R^d; \R)$ and $G$ satisfy \cref{assum-bip} (ii). Given any $N\in \N$, let $\nu_N \in \M_n^{\bxi}$ be a minimizer of the functional $\nu \mapsto D_{\text{KL}}(\nu || \mu_{N}^{\bdeta})$ within $\M_n^{\bxi}$. Let $\nu_N = \sum_{i=1}^n \alpha^i_N \nu^i_N$ with $\nu^i_N = N(m^i_N, \frac{1}{N} \bSig^i_N)$. Then it holds that as $N\gt \infty$
$$
m^i_N \gt x^\dagger_i, \bSig^i_N \gt \left(DG(x^\dagger_i)^T DG(x^\dagger_i)\right)^{-1}\text{ and } \alpha^i_N \gt \frac{\left[\det D G(x^\dagger_i)\right]^{-1} \cdot e^{-V_0(x^\dagger_i)}}{\sum_{j=1}^n 
\left[\det D G(x^\dagger_j)\right]^{-1} \cdot e^{-V_0(x^\dagger_j)}
}.$$
\end{theorem}

\cref{thm:bip} (i) states that the measure $\mu_N^{\bdeta}$ is asymptotically Gaussian when certain uniqueness and stability properties hold in the inverse problem. 
Moreover, in this case, the asymptotic Gaussian distribution is fully determined by the truth and the forward map, and is independent of the prior. 
 In the case where the uniqueness fails, but the data only corresponds to a finite number of unknowns, \cref{thm:bip} (ii) demonstrates that the measure $\mu_N^{\bdeta}$ is asymptotically a Gaussian mixture, with each Gaussian mode independent of the prior. However, prior beliefs affect the proportions of the individual Gaussian components within the mixture; more precisely, the un-normalized weights of each Gaussian mode are
proportional to the values of the prior evaluated at the corresponding unknowns. 
 \begin{remark}
In general, when $\{\eta_i\}_{i\in \N}$ is a sequence of fixed realizations of the normal distribution, \cref{thm:bip} does not hold for the measure $\mu_N^{\bdeta}$ defined in \cref{eq:postdist2} in the large data case. However, we will show that $D_{\text{KL}}(\nu_N || \mu^{\bdeta}_N)$ will converge to zero in some average sense; see \cref{thm:expdkl}.
 \end{remark}

\subsection{A Bernstein-Von Mises type result}


The asymptotic Gaussian phenomenon in \cref{thm:bip} is very much in
the same spirit as the celebrated Bernstein-Von Mises (BvM) theorem \cite{V00}. 
This theorem asserts that for a certain class of regular priors, the posterior distribution converges to a Gaussian distribution, independently of the prior, as the sample size tends to infinity. Let us state the Bernstein-Von Mises theorem more precisely in the i.i.d case. Consider observing a set of i.i.d samples $\bX^N := \{X^1,X^2,\cdots, X^N\}$, where $X^i$ is drawn from distribution $P_{\theta}$, indexed by an unknown parameter $\theta\in \Theta$. Let $P^N_\theta$ be the law of $\bX^N$. Let $\Pi$ be the prior distribution on $\theta$ and denote by $\Pi(\cdot | \bX^N)$ the resulting posterior distribution. The Bernstein-Von Mises Theorem is concerned with the behavior of the posterior $\Pi(\cdot | \bX^N)$ under the frequentist assumption that $X^i$ is drawn from some true model $P_{\theta_0}$. A standard finite-dimensional BvM result (see e.g. \cite[Theorem 10.1]{V00}) states that, under certain conditions on the prior $\Pi$ and the model $P_{\theta}$, as $N \gt \infty$
\be\label{eq:bvm1}
d_{TV}\left(\Pi(\theta| \bX^N), N\left(\hat{\theta}_N, \frac{1}{N} I^{-1}_{\theta_0}\right)\right) \xrightarrow{P_{\theta_0}^N} 0
\en
where $\hat{\theta}_N$ is an efficient estimator for $\theta$, $I_{\theta}$ is the Fisher information matrix of $P_\theta$ and $d_{TV}$ represents the total variation distance. As an important consequence of the BvM result, Bayesian credible sets are asymptotically equivalent to frequentist confidence intervals. Moreover, it has been proved that the optimal rate of convergence in the Bernstein-Von Mises theorem is $O(1/\sqrt{N})$; see, for instance, \cite{LeCam73, HM76}. This means that for any $\delta > 0$, there exists $M = M(\delta) > 0$ such that
\be\label{eq:bvm2}
P_{\theta_0}^N \left(\bX^N : d_{TV}\left(\Pi(\theta| \bX^N), N\left(\hat{\theta}_N, \frac{1}{N} I^{-1}_{\theta_0}\right)\right) \geq M\frac{1}{\sqrt{N}}\right) \leq \delta
\en

Unfortunately, BvM results like \cref{eq:bvm1} and \cref{eq:bvm2} do not fully generalize to infinite dimensional spaces, see counterexamples in \cite{F99}. Regarding the asymptotic frequentist properties of posterior distributions in nonparametric models, various positive results have been obtained recently, see e.g. \cite{GGV00,SW01,KVZ11,L11,CN13,CN14}. For the convergence rate in the nonparametric case, we refer to \cite{GGV00,SW01,CN13}. 

In the remainder of the section, we prove a Bernstein-Von Mises type result for the posterior distribution $\mu^{\bdeta}_N$ defined by \cref{eq:postdist1} and \cref{eq:postdist2}.
If we view the observational noise $\eta$ and $\eta_i$ appearing in the data
as random variables, 
then the posterior measures appearing become random probability measures. Furthermore, 
exploiting the randomness of the $\eta_i$, we claim that the posterior distribution 
in the large date case can be rewritten in the form of the small noise case. Indeed, 
by completing the square, 
we can write the expression \cref{eq:postdist2} as
\be
\mu_N^{\bdeta}(dx) 
= \frac{1}{\overline{Z}_{N,2}^{\bdeta}}\exp\left(-\frac{N}{2}|G(x^\dagger) - G(x) + \frac{1}{N}\sum_{i=1}^N \eta_i|^2\right) dx
\en
Observe that $\mathcal{L}\left(\frac{1}{N} \sum_{i=1}^N \eta_i\right) = \mathcal{L}(\frac{1}{\sqrt{N}}\eta) = N(0, \frac{1}{N}\bI_d)$ due to the normality assumptions on $\eta$ and $\eta_i$. As a consequence it makes no difference which formulation is chosen
when one is concerned with the statistical dependence of $\mu^{\bdeta}_N$ on the law of 
$\bdeta$. For this reason, we will only prove the Bernstein-Von Mises result for $\mu^{\bdeta}_N$ given directly in the form \cref{eq:postdist1}.

For notational simplicity, we write the noise level $\sqrt{\eps}$ in place of $\frac{1}{\sqrt{N}}$ and consider random observations $\{y_\eps\}$, generated from a truth $x^\dagger$ and normal noise $\eta$, i.e.
$$
y_\eps = G(x^\dagger) + \sqrt{\eps} \eta.
$$
Given the same prior defined as before, we obtain the posterior distribution
$$
\bea
\mu^\eta_\eps (dx) & = \frac{1}{Z_{\mu, \eps}^{\eta}} \exp\left(-\frac{1}{2\eps} |y_\eps - G(x)|^2 - V_0(x) \right) dx \\
& =  \frac{1}{Z_{\mu,\eps}^{\eta}} \exp\left(-\frac{1}{2\eps} |G(x^\dagger)- G(x) + \sqrt{\eps} \eta|^2 - V_0(x) \right) dx.
\ena
$$
 For any fixed $\eta$, let $\nu_\eps^\eta$ be the best Gaussian measure which minimizes the Kullback-Leibler divergence $D_{\text{KL}}(\nu || \mu_\eps^\eta)$ over $\A$. For ease of calculations, from now on we only consider the rate of convergence under \cref{assum-bip} (i); the other case can be dealt with in the same manner, see \cref{rem:expdkl}. 
The main result is as follows.

\begin{theorem}\label{thm:expdkl}
There exists $C > 0$ such that 
\be\label{eq:expdkl}
\E^{\eta} D_{\text{KL}}(\nu_\eps^\eta || \mu_\eps^\eta) \leq C\eps
\en
as $\eps \lgt 0$.
\end{theorem}

With the help of Pinsker's inequality \cref{ieq:pinsker} as well as the Markov inequality, one can derive the following BvM-type result from \cref{thm:expdkl}.
\begin{corollary}\label{cor:bvm}
For any $\delta > 0$, there exists a constant $M = M(\delta)  > 0$ such that
\be\label{eq:bvm3}
\P^\eta \left( \eta:  d_{\text{TV}}(\mu_\eps^\eta, \nu_\eps^\eta) \geq M\sqrt{\eps}\right) \leq \delta
\en
when $\eps \lgt 0$.
\end{corollary}


\begin{remark}

\item[(i)] Because of the statistical equivalence of posterior measures in the limit of large data size and  small noise, the posterior measure $\mu_N^{\bdeta}$ in the large data case (given by \eqref{eq:postdist2}) has the same convergence rate as \eqref{eq:bvm3}, namely, for any $\delta > 0$, there exists a constant $M = M(\delta)  > 0$ such that
\be\label{eq:bvm4}
\P^{\bdeta} \left( \bdeta:  d_{\text{TV}}(\mu_N^{\bdeta}, \nu_N^{\bdeta}) \geq M/\sqrt{N}\right) \leq \delta
\en
as $N \gt \infty$. This recovers the optimal rate of convergence for the posterior as proved for statistical models, see \eqref{eq:bvm2}.

\item[(ii)] For a fixed realization of the noise $\eta$, we have shown in \cref{thm:bip} (i) that $D_{\text{KL}}(\nu_N || \mu^{\bdeta}_N)\lgt 0$ as $N \gt \infty$. In fact, by following the proof of the Laplace method, one can prove that $D_{\text{KL}}(\nu_N || \mu^{\bdeta}_N) = \mathcal{O}(1/ \sqrt{N})$. However, we obtain the linear convergence rate in \cref{eq:expdkl} (with $\eps$ replacing $1/N$)  by utilizing the symmetric cancellations in the evaluation of Gaussian integrals.
\end{remark}

To prove \cref{thm:expdkl}, we start with an averaging estimate for the logarithm of the normalization constant $Z^\eta_{\mu, \eps}$. 

\begin{lemma}\label{lem:expnorm}
\be\label{eq:explogZ}
\E^\eta \log Z_{\mu,\eps}^\eta \leq \frac{d}{2} \log (2\pi\eps) -V_0(x^\dagger) + \log \det DG(x^\dagger)+ r_\eps
\en
where $r_\eps \leq C\eps$ for some $C > 0$ independent of $\eps$. 
\end{lemma}
\begin{proof}
Take a constant $\gamma \in (0, \frac{1}{2})$. We write $\E^\eta \log Z_{\mu,\eps}^\eta$ as the sum
$$
\E^\eta \log Z_{\mu,\eps}^\eta = \E^\eta \left(\log Z_{\mu,\eps}^\eta \mathbf{1}_{|\eta| \leq \eps^{-\gamma}}\right) +  \E^\eta \left(\log Z_{\mu,\eps}^\eta \mathbf{1}_{|\eta| \geq \eps^{-\gamma}}\right)=:  I_1 + I_2. 
$$
We first find an upper bound for $I_2$. By definition,
$$
\bea
Z_{\mu,\eps}^\eta &= \int_{\R^d}\exp\left(-\frac{1}{2\eps} |G(x^\dagger)- G(x) + \sqrt{\eps} \eta|^2 - V_0(x) \right) dx\\
& \leq \int_{\R^d} e^{-V_0(x)}dx  = Z_0.
\ena
$$
It follows that 
$$
I_2 \leq \log Z_0 \cdot P^\eta(\eta: |\eta| \geq \eps^{-\gamma}) \leq \log Z_0\cdot e^{-\eps^{-2\gamma}}.
$$
For $I_1$, we need to estimate $Z^\eta_{\mu, \eps}$ under the assumption that $|\eta|\leq \eps^{-\gamma}$. Thanks to the condition (i) on $G$, when $\eps \ll 1$ there exists a unique $m^\dagger_{\eps, \eta}$ such that $G(m^\dagger_{\eps, \eta}) = G(x^\dagger)  + \sqrt{\eps}\eta$. Moreover, denoting by $H$ the inverse of $G$ in the neighborhood of $G(x^\dagger)$, we get from Taylor expansion that
\be\label{eq:mdagger0}
m^\dagger_{\eps, \eta} = x^\dagger + DH(G(x^\dagger)) \sqrt{\eps}\eta + \eps \sum_{|\alpha|=2} \partial_{\alpha} H(\xi G(x^\dagger) + (1 - \xi)\sqrt{\eps}\eta) \eta^\alpha 
\en
with some $\xi \in (0,1)$. Thanks to the smoothness assumption on $G$, the function $H$ is differentiable up to the fourth order and hence the coefficients in the summation are uniformly bounded. Moreover, noting that $DH(G(x^\dagger)) = DG(x^\dagger)^{-1}$, we obtain
\be\label{eq:mdagger1}
m^\dagger_{\eps, \eta} = x^\dagger + DG(x^\dagger)^{-1} \sqrt{\eps}\eta + \eps R_\eps(\eta)
\en
where $\limsup_{\eps \lgt 0}|R_\eps(\eta)| \leq C|\eta|^2$ for some positive $C$ which is independent of $\eps$ and $\eta$. Next, according to the proof of \cref{lem:normconst}, given any sufficiently small $\delta > 0$, we can write $Z_{\mu, \eps}^\eta = I_{\eps}^{\delta,\eta} + J_{\eps}^{\delta,\eta}$ where $|J_{\eps}^{\delta,\eta}| \leq Ce^{-\frac{C}{\eps}}$ with some $C > 0$ independent of $\eta$ and 
$$
I_{\eps}^{\delta,\eta} = \int_{B^{\delta,\eta}_{\eps}}\exp\left(-\frac{1}{2\eps} |G(x^\dagger)- G(x) + \sqrt{\eps} \eta|^2 - V_0(x) \right) dx
$$
with $B^{\delta,\eta}_{\eps} := B(m^\dagger_{\eps, \eta}, \delta)$. Now we seek bounds for $I_{\eps}^{\delta,\eta}$. Thanks to \cref{assum-bip} (i) and the fact that $m^\dagger_{\eps,\eta} \gt x^\dagger$, $G$ is a $C^1$-diffeomorphism in the neighborhood of $m^\dagger_{\eps,\eta}$. Therefore there exist positive constants $\delta_1 < \delta_2$ depending only on $\delta$ such that $B\left(G(m^\dagger_{\eps,\eta}),\delta_1\right) \subset G(B^{\delta,\eta}_{\eps}) \subset B\left(G(m^\dagger_{\eps,\eta}),\delta_2\right)$. After applying the transformation $x \mapsto H(x)$ in evaluation of the integral $I_\eps^{\delta,\eta}$, we get
$$
\tilde{I}^{\delta_1, \eta}_{\eps} 
\leq I_{i,\eps}^{\delta,\eta} \leq 
\tilde{I}^{\delta_2, \eta}_{\eps} 
$$
where 
$$
\tilde{I}^{\delta, \eta}_{\eps}  := \int_{B(0,\delta)} \exp\left(-\frac{1}{2\eps}|y|^2 - V_0 \circ H(y + G(m^\dagger_{\eps,\eta}))\right) \det(DH(y + G(m^\dagger_{\eps,\eta}))dy.
$$
In order to estimate $\tilde{I}^{\delta, \eta}_{\eps}$, in $B(0, \delta)$ with some small $\delta$ we define two auxiliary functions by setting
$$
f_{\eps, \eta}(\cdot) := \exp(-V_0\circ H(\cdot + G(m^\dagger_{\eps,\eta}))) \det(DH(\cdot + G(m^\dagger_{\eps,\eta}))$$
and 
$$
L(\cdot) := \exp(-V_0\circ H(G(\cdot))) \det(DH(G(\cdot))  = \exp(-V_0(\cdot))/ \det(DG(\cdot)).
$$ 
It is worthy to note that within the ball $B(0,\delta)$, all derivatives up to second order of $f_{\eps, \eta}$ as well as of $L$ can be bounded uniformly with respect to sufficiently small $\eps$ and $\eta$ such that $|\eta| \leq \eps^{-\gamma}$.
Taking the equation \cref{eq:mdagger1} into account, we can expand $L$ near $m^\dagger$ to get that
\be\label{eq:feps-eta1}
\bea
& f_{\eps, \eta}(0)  = L (m^\dagger_{\eps,\eta}) \\
& = L(x^\dagger) + \nabla L(x^\dagger)^T (m^\dagger_{\eps,\eta} - x^\dagger) + \frac{1}{2}(m^\dagger_{\eps,\eta} - x^\dagger)^T\nabla^2 L(\theta x^\dagger + (1- \theta)m^\dagger_{\eps,\eta})(m^\dagger_{\eps,\eta} - x^\dagger) \\
& = \frac{\exp(-V_0(x^\dagger) )}{\det(DG(x^\dagger))}+ \eps^{\frac{1}{2}}\nabla L(x^\dagger)^T  DG(x^\dagger)^{-1}\eta + r_{1, \eps,\eta}
\ena
\en
with some $\theta \in (0,1)$ and the residual $|r_{1, \eps, \eta}| \leq C \eps |\eta|^2$ for some $C> 0$. 
Moreover, for any $y\in B(0, \delta)$,
\be\label{eq:feps-eta2}
f_{\eps, \eta}(y) = f_{\eps, \eta} (0) + \nabla f_{\eps, \eta}(0)^T y + \frac{1}{2}y^T \nabla^2 f_{\eps, \eta}(\xi y) y
\en
for some $\xi  = \xi(y)\in (0,1)$. Then it follows from \cref{eq:feps-eta1} and \cref{eq:feps-eta2} that
\be
\bea
\tilde{I}^{\delta, \eta}_{\eps} &= \int_{B(0,\delta)}  \exp(-\frac{1}{2\eps}|y|^2) f_{\eps,\eta}(y) dy\\
& = \eps^{\frac{d}{2}}\int_{B(0,\eps^{-\frac{1}{2}}\delta)} \exp(-\frac{1}{2}|y|^2) f_{\eps, \eta}(\eps^{\frac{1}{2}}y) dy\\
& = \eps^{\frac{d}{2}}\left(f_{\eps, \eta}(0) \int_{B(0, \eps^{-\frac{1}{2}}\delta)} \exp(-\frac{1}{2} |y|^2) dy + \frac{\eps}{2} \int_{B(0,\eps^{-\frac{1}{2}}\delta)} \exp(-\frac{1}{2}|y|^2) y^T \nabla^2 f_{\eps, \eta}(\xi y) y dy \right)\\
& = (2\pi\eps)^{\frac{d}{2}} \left(\frac{\exp(-V_0(x^\dagger))}{\det(DG(x^\dagger))} +  \nabla L(x^\dagger)^T DG(x^\dagger)^{-1} \sqrt{\eps}\eta + r_{2,\eps, \eta} \right)
\ena
\en
with $|r_{2,\eps, \eta}| \leq C\eps |\eta|^2$. Notice that the linear term in the expansion \cref{eq:feps-eta2} vanishes from the second line to the third line because the region of integration is symmetric with respect to the origin; the final equality holds because we have counted the exponentially decaying Gaussian integral outside of the ball $B(0, \eps^{-\frac{1}{2}}\delta)$ in the residual $r_{2,\eps, \eta}$. Hence we obtain that for $|\eta| \leq \eps^{-\gamma}$ and $\eps$ small enough
$$
I^\delta_{\eps, \eta} = (2\pi\eps)^{\frac{d}{2}} \left(\frac{\exp(-V_0(x^\dagger))}{\det(DG(x^\dagger))} +  \eps^{\frac{1}{2}}\nabla L(x^\dagger)^T DG(x^\dagger)^{-1} \eta + r_{2,\eps, \eta} \right)
$$
with $|r_{2,\eps, \eta}| \leq C\eps |\eta|^2$. As a result, $Z_{\mu,\eps}^{\eta}$ satisfies the same bound as above. Then by using the Taylor expansion of the log function, one obtains that
$$
\log Z^\eta_{\mu,\eps} = \log\left( \frac{(2\pi\eps)^{\frac{d}{2}}\exp(-V_0(x^\dagger))}{\det(DG(x^\dagger))}\right) + \eps^{\frac{1}{2}} p^T \eta + r_{3, \eps, \eta}
$$
where $p$ is vector depending only on $L, G, V_0$ and $x^\dagger$ and $|r_{3,\eps, \eta}| \leq C\eps |\eta|^2$. This implies that when $\eps$ is sufficiently small,
$$
I_1 = \E^{\eta}\left( \log Z_{\mu,\eps}^\eta \mathbf{1}_{|\eta| \leq \eps^{-\gamma}}\right)  =\frac{d}{2} \log (2\pi\eps) -V_0(x^\dagger) + \log \det DG(x^\dagger)+ r_\eps.
$$
with $|r_\eps| \leq C\eps$. Again the first order term $\eps^{\frac{1}{2}}p^T\eta$ vanishes because of the symmetry of the integration region; the bound $|r_\eps| \leq C\eps$ follows from the bound for $r_{3,\eps,\eta}$ and the Gaussian tail bound. This completes the proof.
\end{proof}

\begin{proof}[Proof of \cref{thm:expdkl}]
We prove the theorem by constructing a family of Gaussian measures $\{\overline{\nu}_\eps^\eta\}$ such that 
\be\label{eq:expdkl-1}
\E^{\eta} D_{\text{KL}}(\overline{\nu}_\eps^\eta || \mu_\eps^\eta) \leq C\eps
\en 
for some $C > 0$. Then the theorem is proved by the optimality of $\nu_{\eps, \eta}$. Recall that $m^\dagger_{\eps, \eta}$ is defined by \cref{eq:mdagger0}. Fixing $\gamma\in (0,\frac{1}{2})$, we define $\overline{\nu}_\eps^\eta = N(\overline{m}_{\eps, \eta}, \overline{\bSig}_{\eps, \eta})$ with $\overline{m}_{\eps, \eta}$ defined by
$$
\overline{m}_{\eps, \eta} = \begin{cases}
m^\dagger_{\eps, \eta} & \text{ if } |\eta| \leq \eps^{-\gamma},\\
x^\dagger & \text{ otherwise }
\end{cases}
$$
and that $\overline{\bSig}_{\eps,\eta} =  \big(DG(\overline{m}_{\eps, \eta})^TDG(\overline{m}_{\eps, \eta})\big)^{-1}$. Clearly, 
when $\eps$ is small enough, $\overline{m}_{\eps, \eta}$ admits an expansion 
similar to \cref{eq:mdagger0}. As a consequence, there exist positive constants $C_1, c_2, C_2$ which are independent of $\eta$, such that $\limsup_{\eps \lgt 0} |\overline{m}_{\eps, \eta}| \leq C_1$ and $c_2 \leq \liminf_{\eps \lgt 0} \lambda_{\text{min}} (\overline{\bSig}_{\eps,\eta}) < \limsup_{\eps \lgt 0} \Tr(\bSig_\eps) \leq C_2$ hold for all $\eta$. With the above choice for $(\overline{m}_{\eps, \eta}, \overline{\bSig}_{\eps,\eta})$, an application of \cref{lem:dkl-asym2} with $V_1^\eps(x) = \frac{1}{2}|G(x^\dagger) - G(x) + \sqrt{\eps} \eta|^2$ and $V_2(x) = V_0(x)$ yields that
\be\label{eq:expdkl2}
D_{\text{KL}}(\overline{\nu}_\eps^\eta || \overline{\mu}_\eps^\eta) = V_0(\overline{m}_{\eps, \eta}) -\frac{d}{2}\log (2\pi\eps) + \log \det DG(\overline{m}_{\eps, \eta})+ \log Z_{\mu, \eps}^{\eta} + r_\eps
\en
where $r_\eps \leq C\eps$ with $C = C(C_1,c_2, C_2, M_V)$. By the definition of $\overline{m}_{\eps, \eta}$ and the expansion \cref{eq:mdagger0}, it follows from the Taylor expansion for the function $x \mapsto V_0(x) + \frac{1}{2} \log \det DG(x)$ that when $|\eta| \leq \eps^{-\gamma}$ and $\eps$ is small enough,
\be\label{eq:rhologdet}
V_0(\overline{m}_{\eps, \eta}) + \log \det DG(\overline{m}_{\eps, \eta}) = V_0(x^\dagger) + \log \det DG(x^\dagger) + \sqrt{\eps} q^T \eta + \tilde{r}_{\eps,\eta}
\en
with some $q\in \R^d$ and $|\tilde{r}_{\eps,\eta}| \leq C\eps$ for some $C > 0$. Then the estimate \cref{eq:expdkl-1} follows, by taking the expectation of \cref{eq:expdkl2} and using the equation \cref{eq:rhologdet} and \cref{lem:dkl-asym2}. 
\end{proof}

\begin{remark}\label{rem:expdkl}
\cref{thm:expdkl} proves the rate of convergence with the assumption that $G$ satisfies \cref{assum-bip} (i). However, the convergence rate remains the same when \cref{assum-bip} (ii) is fulfilled, and when the best Gaussian measure is replaced by the best Gaussian mixture. 
\end{remark}

\subsection{Comparison with Classical BvM Results}

We would like to make comparisons between our BvM result for Bayesian inverse problems and classical finite dimensional BvM results for general statistical models \cite{GR03, HM76}.

\begin{itemize}
\item {\bf Assumption.} In the classical framework of Bayesian inferences,  the posterior converges to a Gaussian in the total variation distance (with optimal rate) under the typical assumption that the likelihood function is $C^3$ and that the Fisher information matrix is non-degenerate; see e.g. \cite[Theorem 1.4.2]{GR03} and \cite[Section 4]{HM76}. The asymptotic covariance of the limiting Gaussian is given by the inverse of the Fisher information matrix. In the Bayesian inverse problem setting, we improve the convergence to the stronger sense of KL-divergence, but at the expense of requiring higher differentiability ($C^4$) on the forward map $G$. Moreover, the matrix product $DG^T DG$ takes the place of the Fisher information matrix in the asymptotic covariance, where $DG$ is invertible because of \cref{assum-bip}.

\item {\bf Multimodal Distribution.} The proposed KL-approximation framework allows us to prove the convergence of a multimodal probability measure to a mixture of Gaussian measures. The limiting KL-discrepency between the target measure and the Gaussian approximation is characterized explicitly as a sum of two relative entropies, see \cref{thm:gamma-2}. In addition, in this case the prior does not disappear in the limit and its influence on the posterior is reflected in the weighted coefficients in the Gaussian mixture. 
 To the best of our knowledge, such results
have not been stated in the statistical literature. 

\item {\bf Proof.} Both our proof and classical proofs for the finite dimensional BvM theorems are essentially based on the local Taylor expansion of the posterior around the truth. But the proofs are carried out in different ways.  Classical BvM results in the TV-distance are usually proved by first expanding the posterior density around the maximum likelihood estimator (MLE), which requires tracking the normalization constant,  and then applying the local asymptotic normality of MLE and LeCam's contiguity arguments to obtain the convergence of the posterior. Our proof, instead, takes advantage of the special formulation of the KL-divergence, i.e. the separation of the normalization constant from the log density, thereby  reducing the convergence proof to establishing
precise estimates on the normalization constant (see \cref{lem:expnorm}). 

\end{itemize}

\section{Conclusions}\label{sec:conclusion}

We have studied a methodology widely used in applications, yet little analyzed,
namely the approximation of a given target measure by a Gaussian, or by
a Gaussian mixture. We have employed relative entropy as a measure of goodness
of fit. Our theoretical framework demonstrates the existence of
minimizers of the variational problem, and studies their asymptotic
form in a relevant small parameter limit where the measure concentrates;
the small parameter limit is studied by use of tools from $\Gamma$-convergence.
In the case of a target with asymptotically unimodal distribution the $\Gamma$-limit
demonstrates perfect reconstruction by the approximate single Gaussian method in
the measure concentration limit; 
and in the case of multiple modes it quantifies the errors resulting
from using a single mode fit. Furthermore the Gaussian mixture is shown to
overcome the limitations of a single mode fit, in the case of target measure
with multiple modes. These ideas are exemplified in the analysis of
a Bayesian inverse problem in the small noise or large data set limits,
and connections made to the Bernstein-von Mises theory from asymptotic
statistics.

The BvM theorem of this paper is essentially still parametric. 
A natural interesting future direction would be to study 
infinite-dimensional statistical models \cite{nickl}. 
In particular it would be interesting to apply our measure approximation approach from $\Gamma$-convergence to understand the BvM phenomenon of infinite dimensional non-linear Bayesian inverse problems. In our finite dimensional setting, the inverse problem of interest is 
essentially well-posed since we assume that both $G$ and $DG$ are invertible, so the only ill-posedness comes from the lack of uniqueness. However, for infinite dimensional inverse problems, the degree of ill-posedness (mild/severe) has a big influence on the precise statement of the BvM theorem. Understanding of this issue requires delicate quantitative stability estimates for the underlying inverse problem.  The recent paper \cite{lu2017bernstein} proved a BvM result for high dimensional non-linear inverse problems where dimension of the unknown parameter increases with the decreasing noise level. However, it remains an open problem whether the BvM theorem holds for genuinely infinite dimensional non-linear inverse problems. We will address this problem in future work.

\bibliographystyle{siamplain}
\bibliography{klsiamuq}
\end{document}